\documentclass[11pt, a4paper]{article}
\usepackage{amsmath,amssymb,amsthm,amsfonts,latexsym,graphicx,subfigure}
\usepackage[all,import]{xy}
\usepackage{color}
\usepackage[numbers,sort&compress]{natbib}
\usepackage{indentfirst}
\usepackage{fancyhdr}
\usepackage{bbm}
\usepackage{accents,cases}
\usepackage{algorithm}
\usepackage{algorithmic}
\usepackage{multirow}
\usepackage{enumerate}

\usepackage{array}
\newcommand{\PreserveBackslash}[1]{\let\temp=\\#1\let\\=\temp}
\newcolumntype{C}[1]{>{\PreserveBackslash\centering}p{#1}}
\newcolumntype{R}[1]{>{\PreserveBackslash\raggedleft}p{#1}}
\newcolumntype{L}[1]{>{\PreserveBackslash\raggedright}p{#1}}

\DeclareMathOperator*{\Ext}{\ensuremath{Ext}}

\makeatletter
\def\wbar{\accentset{{\cc@style\underline{\mskip8mu}}}}

\makeatother

\numberwithin{equation}{section}

\theoremstyle{plain}
\newtheorem{theorem}{Theorem}[section]
\newtheorem{defn} [theorem] {Definition}
\newtheorem{lemma} [theorem] {Lemma}
\newtheorem{remark}[theorem]{Remark}

\newtheorem{cor}[theorem]{Corollary}
\newtheorem{pro}[theorem]{Proposition}

\begin{document}
\bibliographystyle{unsrt}
\title{Algebraic Surfaces of General Type  with $p_g=q=1$ and Genus 2 Albanese Fibrations}
\author{Songbo Ling}
\date{}
\renewcommand{\thefootnote}{\fnsymbol{footnote}}
\maketitle
 
\footnotetext{This work was completed at Universit\"at Bayreuth under the financial support of China Scholarship Council ``High-level university graduate program''.}

\begin{abstract}
In this paper, we study algebraic surfaces of general type with $p_g=q=1$ and genus 2 Albanese  fibrations.

We first study the examples of surfaces with $p_g=q=1, K^2=5$ and genus 2 Albanese fibrations constructed by Catanese   using singular bidouble covers of $\mathbb{P}^2$. We  prove that these surfaces  give an irreducible  and connected   component of   $\mathcal{M}_{1,1}^{5,2}$,  the Gieseker moduli space of surfaces of general type with $p_g=q=1, K^2=5$  and genus 2 Albanese fibrations.

Then by constructing surfaces with $p_g=q=1,K^2=3$ and a genus 2 Albanese  fibration such that the number of the summands of the direct image of the bicanonical sheaf (under the Albanese map) is 2, we give a negative answer to a question of Pignatelli.
\end{abstract}

\section{Introduction}
Algebraic surfaces of general type with $p_g=q=1$ have attracted interest of  many authors since they are irregular surfaces of general type with the lowest geometric genus. For these surfaces, one has $2\leq K^2\leq 9$ and  the Albanese map is a genus $g$ fibration over an elliptic curve.
 By the  results of Mo\v{\i}\v{s}ezon \cite{Moi65},  Kodaira \cite{Kod68} and Bombieri \cite{Bom73},   these surfaces belong to a finite number  of families.  
Since $g$ is a differentiable invariant (\cite{CP06} Remark 1.1),  such surfaces   with different $g$ belong to different connected components of the moduli space.

In this paper, we restrict ourselves to the case $g=2$. 
By a result of Xiao \cite{Xiao85}, one has $K^2\leq 6$.   The case $K^2=2$ has been   accomplished by Catanese \cite{Cat81} and Horikawa \cite{Hor81} independently: the moduli space of such surfaces is irreducible of dimension 7.  The case $K^2=3$ has been studied   by Catanese-Ciliberto \cite{CC91, CC93} and completed by Catanese-Pignatelli  \cite{CP06}: the moduli space of such surfaces consists of three  5-dimensional  irreducible and connected components.

In the case $K^2=4$, there are many examples  (e.g.  Catanese \cite{Cat98}, Rito \cite{Rit1}, Polizzi \cite{Pol09}, Frapporti-Pignatelli \cite{FP15} and Pignatelli \cite{Pig09}).  In particular, Pignatelli \cite{Pig09} found 8 irreducible components of the moduli space of surfaces of general type  under the assumption that  the direct image of the bicanonical sheaf under the Albanese map  is a direct sum of three  line bundles.  

However, there are very few examples when $K^2>4$. In fact,   for  surfaces with $p_g=q=1$, $K^2=5$, $g=2$, the only   known examples   are constructed by Catanese \cite{Cat98} and Ishida \cite{Ish05};  for surfaces with $p_g=q=1$, $K^2=6$,  $g=2$, no example is known.

In this paper, we first analyze the examples constructed by Catanese in \cite{Cat98} Example 8 and prove the following
\begin{theorem}
	\label{theorem 5,2}
	The surfaces constructed by Catanese constitute  a 3-dimensional  irreducible and connected component of $\mathcal{M}_{1,1}^{5,2}$.
\end{theorem}
The idea of the proof for Theorem \ref{theorem 5,2}  is the following.

First we show that  
a general surface in each of  the two families  (in \cite{Cat98} Example 8, case I and case II) is a smooth bidouble cover of the Del Pezzo surface of degree 5. Moreover,  we prove that  the two  families  are equivalent up to an automorphism of this Del Pezzo surface.  
Hence the images of the two families coincide as an irreducible subset  $\mathcal{M}$ in  $\mathcal{M}_{1,1}^{5,2}$.

Then using Catanese's theorem \cite{Cat84} on  deformations of smooth bidouble covers and  a method of Bauer-Catanese \cite{BC13},  we calculate $h^1(T_S)$ for a general surface in this family and show  that it is equal to the dimension of  $\mathcal{M}$ (which is 3).    By studying  the limit surface in the family, we show that $\mathcal{M}$ is a Zariski closed subset of $\mathcal{M}_{1,1}^{5,2}$, hence $\mathcal{M}$ is an irreducible component of $\mathcal{M}_{1,1}^{5,2}$.

By studying the deformation of the branch curve of the double cover $S\rightarrow \mathcal{C}\subset \mathbb{P}(V_2)$ (where $V_2=f_*\omega_{S/B}^{\otimes 2}$ and  $\mathcal{C}$ is the conic bundle in Catanese-Pignatelli's structure theorem for genus 2 fibrations), we show that $\mathcal{M}$ is an analytic open subset of  $\mathcal{M}_{1,1}^{5,2}$.   Therefore  $\mathcal{M}$ is a  connected component of $\mathcal{M}_{1,1}^{5,2}$.

\vspace{3ex}
Topological and deformation invariants play an important role in studying the moduli spaces of algebraic surfaces. For surfaces $S$ of general type with $p_g=q=1$, Catanese-Ciliberto (cf. \cite{CC91} Theorems 1.2 and  1.4)  proved   that the the number $\nu_1$ of direct summands of $f_*\omega_S$ (where $f$ is the Albanese fibration of $S$ and $\omega_S$ is the canonical sheaf of $S$)  is a topological invariant. After that Pignatelli (cf. \cite{Pig09} p. 3) asked: is the number $\nu_2$ of direct summands of  $f_*\omega_S^{\otimes 2}$ a deformation or a topological invariant?

In the last part of this thesis  we give a negative answer to Pignatelli's question, i.e.
\begin{theorem}
	\label{theorem number of V2}
	The number $\nu_2$ is  not a deformation invariant, thus it is  not a topological invariant, either.
\end{theorem}
The idea is to show that $\mathcal{M}_{II}^{3,2}$ is nonempty, where $\mathcal{M}_{II}^{3,2}$  is the subspace of $\mathcal{M}_{1,1}^{3,2}$ corresponding to surfaces with  $\nu_2=2$ (see \cite{CP06} Definition 6.11). Since Catanese-Pignatelli  (\cite{CP06} Proposition 6.15) showed that $\mathcal{M}_{II}^{3,2}$  cannot contain any irreducible component of $\mathcal{M}_{1,1}^{3,2}$,  this implies  that surfaces with $\nu_2=2$ can be deformed to surfaces with $\nu_2=1$ or $\nu_2=3$. Therefore $\nu_2$ is not a deformation invariant.

\vspace{3ex}
$\mathbf{Notation~  and ~ conventions.}$
Throughout this paper we work over the field $\mathbb{C}$ of complex numbers. We  denote by $S$ a minimal  surface of general type with $p_g=q=1$,  $S'$ the canonical model of $S$ and  $X$
 the Del Pezzo surface of degree 5.

 We denote by $\Omega_S$    the sheaf of holomorphic 1-forms on $S$,  by $T_S:=\mathcal{H}om_{\mathcal{O}_S}(\Omega_S,\mathcal{O}_S)$   the tangent sheaf  of $S$ and    by  $\omega_S:=\wedge^2\Omega_S$    the sheaf of holomorphic 2-forms on $S$. $K_S$ (simply  $K$ if   no confusion) is  the canonical divisor of $S$, i.e.  $\omega_S\cong \mathcal{O}_S(K_S)$.  $p_g:=h^0(\omega_S),  q:=h^0(\Omega_S)$.   $f: S\rightarrow B:=Alb(S)$ is  the Albanese fibration  of $S$ and  $g$ is the genus of the Albanese  fibres. $V_n:=f_*\omega_S^{\otimes n}$.
 
  For an elliptic curve $B$ and a point $p\in B$, we denote   by $E_p(r,1)$  the unique indecomposable  vector bundle of rank $r$ over $B$ with determinant $\det(E_p(r,1))\cong  \mathcal{O}_B(p)$ (see Atiyah   \cite{Ati57} Theorem 7).

We denote by `$\equiv$' the linear equivalence for divisors. 

$\mathcal{M}_{1,1}^{K^2,2}$   denotes  the Gieseker moduli space of  minimal surfaces of general type  with $p_g=q=1, g=2$ and fixed $K^2$.

\section{The two families constructed by Catanese}
In this section, we show that a general surface with  $p_g=q=1, K^2=5,g=2$ in each of the two families  constructed by Catanese (\cite{Cat98} Example 8, case I and case II) is a  smooth bidouble cover of the  Del Pezzo surface $X$ of degree 5. Moreover,  we prove that the two families are equivalent  up to  an automorphism of $X$, which is induced by a Cremona transformation of $\mathbb{P}^2$.

\vspace{3ex}
Recall that the surfaces constructed by Catanese are obtained by desingularization of  bidoubles covers over $\mathbb{P}^2$ with branch curves $(A,B,C)$   ({\em in this section we use $B$ for one of the branch curve, but  in the following sections, $B$ always denotes  the image of the  Albanese map of $S$}).   Let $P_1,P_2,P_3,P_4$ be four points   in general position (i.e. no three points are collinear)  in $\mathbb{P}^2$, then  $A=A_1+A_2+A_3$,  where $A_i$ is the  line passing through $P_4$ and $P_i$;  $B$ consists of a triangle $B_1+B_2+B_3$ with vertices  $P_1,P_2,P_3$ and a conic $B'$ passing through   $P_1,P_2,P_3$;  $C$ is a line.

In case I, $P_4$ does not belong to $B'$  and $C$ is a general line passing through $P_4$;

In case II, $P_4$ belongs to $B'$ and  $C$ goes through none  of the intersection points of $B$ with $A$.

Note that in both cases, the branch curves  $(A,B,C)$ have  the same degrees $(3,5,1)$ and the four points $P_1,P_2,P_3,P_4$ are singularities of type $(0,1,3)$\footnotemark.
\footnotetext{This means that the respective multiplicities of the three branch curves   at the point are $(0,1,3)$.} As Catanese showed, a general surface in each family is a minimal algebraic surface with $p_g=q=1,K^2=5$ and a genus 2 Albanese fibration.

\begin{lemma}
	\label{both are smooth bidouble cover}
	A general surface $S_1$ (resp. $S_2$) in \cite{Cat98} Example 8 case I (resp. case II) is a smooth bidouble cover of the Del Pezzo surface of degree 5.
\end{lemma}
\begin{proof}
	
	Let $\sigma: X\rightarrow \mathbb{P}^2$  be the	blowing up of $\mathbb{P}^2$ at the four points $\{P_i\}_{i=1}^4$.   Then $X$ is the Del Pezzo surface of degree 5   since the four points are a projective basis of $\mathbb{P}^2$.
	
 Denote by $L$ the pull back of a line $l$ in $\mathbb{P}^2$ via $\sigma$. Denote by $E_i$ (resp. $E_i'$) the  exceptional curve lying over $P_i$ $(i=1,2,3,4)$ in case I (resp. in case II), $L_{ij}$ (resp. $L_{ij}'$)  the strict transform of the line $l_{ij}$ passing through $P_i,P_j$ $(i,j\in \{1,2,3,4\}; i\neq j)$ in case I (resp. in case II).  Denote by $C_1$ (resp. $C_2)$ the strict transform of the  line $C$  in case I (resp. case II), and denote by $Q_1$ ( resp. $Q_2)$ the  strict transform of the conic $B'$ contained in the divisor $B$ in case I (resp. case II).
	
In case I,  let  $ D_1=L_{14}+L_{24}+L_{34}, D_2=Q_1+L_{12}+L_{23}+L_{13}+E_{4}, D_3=C_1+E_{1}+E_{2}+E_{3}$ and     $L_1\equiv 3L-E_1-E_2-E_3, L_2\equiv2L-2E_4, L_3\equiv4L-2E_1-2E_2-2E_3-E_4 $. It is easy to see that $2L_i\equiv D_j+D_k$ and $D_k+L_k\equiv L_i+L_j$ for $\{i,j,k\}=\{1,2,3\}$. Moreover,	$D:=D_1\cup D_2\cup D_3$ has normal crossings.  Hence  the effective divisors $D_1,D_2,D_3$ and divisors $L_1,L_2,L_3$ determine a smooth bidouble cover $\pi^1: \hat{S}_1\rightarrow X$. (cf. \cite{Cat84} Proposition 2.3) One checks easily that $\hat{S}_1=S_1$.

	Similarly, in case II, let $D'_1=L'_{14}+L'_{24}+L'_{34},$
	$ D'_2=Q_2+L'_{12}+L'_{13}+L'_{23},$
	$ D'_3=C_2+E'_1+E'_2+E'_3+E'_4$ and  $L'_1\equiv 3L-E'_1-E'_2-E'_3$,   $L'_2\equiv 2L-E'_4$, $L'_3\equiv 4L-2E'_1-2E'_2-2E'_3-2E'_4$.  Then the effective divisors $D'_1,D'_2,D'_3$ and divisors $L'_1,L'_2,L'_3$ determine a smooth bidouble cover  $\pi^2: \hat{S}_2\rightarrow X$. Moreover $\hat{S}_2=S_2$.
\end{proof}

Now we study the transform of  branch curves under a  suitable   Cremona transformation of  $\mathbb{P}^2$.  Denote by $l_{ij}$ $(1\leq i\neq j\leq 4)$ the line passing through $P_i,P_j$.   By abuse of notation, we still denote by $C_1$ (resp. $C_2)$ for the line $C$ in case I (resp. case II), and by  $B_1'$ (resp. $B_2')$ for the conic contained in $B$ in case I (resp. case II).

Since $\{P_i\}_{i=1}^4$ are a projective basis of $\mathbb{P}^2$, we can find  a coordinate system $(x:y:z)$ on $\mathbb{P}^2$ such that  $P_1=(1:0:0), P_2=(0:1:0), P_3=(0:0:1), P_4=(1:1:1)$. Then $l_{23}=\{ x=0\}$, $l_{13}=\{ y=0\}$, $l_{12}=\{ z=0\}$,   $C_1=\{a_1x+a_2y+a_3z=0| a_1\neq 0, a_2\neq0, a_3\neq0, a_1+a_2+a_3=0\}$ and  $B_1'=\{b_1yz+b_2xz+b_3xy=0| b_1\neq 0, b_2\neq0, b_3\neq0,   b_1+b_2+b_3\neq0\}$,

Let $\phi: \mathbb{P}^2\dashrightarrow  \mathbb{P}^2$ be the   Cremona transformation  such that
$\phi: (x:y:z)\mapsto (yz:xz:xy)$.   Then $\phi:$
$P_i \mapsto  l_{jk}$, $ l_{jk} \mapsto P_i$, ($\{i,j,k\}=\{1,2,3\}$);
$P_4\mapsto P_4$.
Note that   $\phi^{-1}(C_1)=\{ a_1yz+a_2xz+a_3xy=0| a_1\neq 0, a_2\neq0, a_3\neq0, a_1+a_2+a_3=0\}$    is a smooth conic containing $P_1, P_2, P_3$ and $P_4$, which is exactly $B_2'$;   $\phi^{-1}(B_1')=\{ b_1x'+b_2y'+b_3z'=0 \}| b_1\neq 0, b_2\neq0, b_3\neq0, b_1+b_2+b_3\neq0\}$ is a line containing none of the four points $P_i (i=1,2,3,4)$, which is exactly $C_2$. Hence under  $\phi$, $C_2\mapsto B_1'$, $B_2'\mapsto C_1$.

Note that  $\phi$ induces a holomorphic automorphism $\Phi$ on   $X$ and   $\Phi$ acts as : $L_{ij} \mapsto E'_{k}$, $E_{k} \mapsto L'_{ij}$ for $\{i,j,k\}=\{1,2,3\}$; $C_1 \mapsto Q_2$, $Q_1 \mapsto C_2$; and $E_4\mapsto E'_4$, $L_{i4} \mapsto L'_{i4}$ for $i \in \{1,2,3\}$.  So  $\Phi(D_1,D_2,D_3)=(D_1',D_3',D_2')$.
Therefore,  we have  the following:

\begin{pro}	
	\label{equivalent of two families}
	The two families of algebraic surfaces   in \cite{Cat98} Example 8 case I and case II are equivalent up to    an automorphism of $X$.
\end{pro}

Since the two families   in \cite{Cat98} Example 8  are equivalent, we only need to study one of them. From now on, we focus on the family in case I. Considering Lemma \ref{both are smooth bidouble cover}, we give the following definition and notation:

\begin{defn}
	\label{definition of M}
	We denote by $M$  the family of minimal surfaces  in  \cite{Cat98} Example 8 case I.  Denote by   $\mathcal{M}$ the image of $M$ in $\mathcal{M}^{5,2}_{1,1}$ and by $\overline{\mathcal{M}}$  the Zariski closure of $\mathcal{M}$ in $\mathcal{M}^{5,2}_{1,1}$.
\end{defn}

From the construction of the family $M$, it is easy to  calculate the dimension of $\mathcal{M}$.
\begin{lemma}
	\label{dimM 1}
	$\mathcal{M}$ is  a 3-dimensional  irreducible subset  of $\mathcal{M}^{5,2}_{1,1}$.
\end{lemma}
\begin{proof}
	$M$ is a 3-parameter irreducible family: no parameter for $\{P_1,P_2,P_3,P_4\}$,   no parameter for $A=A_1+A_2+A_3$ and the triangle $B_1+B_2+B_3$ (since they are determined by $\{P_1,P_2,P_3,P_4\}$),  2 parameters for the conic $B'$ passing though $P_1, P_2, P_3$ and
	1 parameter for the line $C$ passing though $P_4$.
	
	$M$ gives a family of surfaces $S$ endowed with an inclusion $\psi: (\mathbb{Z}/2\mathbb{Z})^2\hookrightarrow Aut(S)$, which determines the bidouble cover $\pi: S\rightarrow X$. Since $Aut(S)$ is a finite group, for a fixed $S$, there are only finite  choices for $\psi$. On the other hand,  there is a biholomorphism $h: (S_1,\psi_1)\xrightarrow{\sim} (S_2,\psi_2)$ if and only if there is a biholomorphic automorphism $h'$ of  $X$  such that the following diagram
	$$\xymatrix{S_1\ar[r]^h \ar[d]^{\pi_1} &S_2\ar[d]^{\pi_2}\\
		X\ar[r]^{h'} &X
	}$$
	commutes. Since  $Aut(X)$ is isomorphic to the symmetric group $\mathfrak{S}_5$ (cf. e.g. \cite{Dol12} Theorem 8.5.8 or  \cite{Cat16} Theorem 67), which is a finite group, we see that   there are only finitely many surfaces in $M$ isomorphic to $S$.
	Therefore, $\mathcal{M}$ is  a 3-dimensional  irreducible subset  of $\mathcal{M}^{5,2}_{1,1}$.
\end{proof}

\section{   $\overline{\mathcal{M}}$ is an irreducible component of $\mathcal{M}^{5,2}_{1,1}$}
Let $S$ be a general surface in $M$.
In this section, we  calculate $h^1(T_S)$ and show that $\overline{\mathcal{M}}$ is an irreducible component of $\mathcal{M}^{5,2}_{1,1}$.  

We use notation in  section 2. Let $\sigma: X\rightarrow \mathbb{P}^2$ be the	blowing up of $\mathbb{P}^2$ at the four   points $P_1,P_2,P_3,P_4$  in general position. Denote by $E_i$ the  exceptional curve lying over $P_i$ $(i=1,2,3,4)$, $L$ the pull back of a line $l$ in $\mathbb{P}^2$ via $\sigma$, $L_{ij}$ the strict transform of the line $l_{ij}$ passing through $P_i,P_j$ $(i,j\in \{1,2,3,4\}; i\neq j)$, $C$ the strict transform of a  line $l_4$ passing through $P_4$, and $Q$ the strict transform of a conic $\bar{Q}$ passing though $P_1, P_2, P_3$.

By Lemma \ref{both are smooth bidouble cover},  $S$ is  a smooth bidouble of $X$ (which we denote by $\pi$)    determined by effective divisors    $ D_1=L_{14}+L_{24}+L_{34}, D_2=Q+L_{12}+L_{23}+L_{13}+E_{4}, D_3=C+E_{1}+E_{2}+E_{3}$ and divisors    $L_1\equiv 3L-E_1-E_2-E_3, L_2\equiv2L-2E_4, L_3\equiv4L-2E_1-2E_2-2E_3-E_4 $. Since $K_X\equiv -3L+E_1+E_2+E_3+E_4$, we see    $K_X+L_1\equiv E_4, K_X+L_2\equiv -L+E_1+E_2+E_3-E_4\equiv -L_{12}+E_3-E_4, K_X+L_3\equiv L-E_1-E_2-E_3\equiv  L_{12}-E_3.$

Since $H^0(T_S)=0$, by  Riemann-Roch, we have  $-\chi(T_S)=h^1(T_S)-h^2(T_S)=10\chi(\mathcal{O}_S)-2K^2=0$. Hence $h^1(T_S)=h^2(T_S)=h^0(\Omega_S\otimes \omega_S)$ by Serre duality. By \cite{Cat84} Theorem 2.16, we have $$H^0(\Omega_S\otimes \omega_S)\cong H^0(\pi_*(\Omega_S\otimes \omega_S))$$
$$=H^0(\Omega_X(logD_1, logD_2, logD_3)\otimes\omega_X)\oplus(\bigoplus_{i=1}^3 H^0(\Omega_X(logD_i) (K_X+L_i)).$$

To calculate $h^1(T_S)$, it suffices to calculate $h^0(\Omega_X(logD_1, logD_2, logD_3)\otimes\omega_X)$ and $h^0(\Omega_X(logD_i) (K_X+L_i))$ $(i=1,2,3)$.
The first one is easy to calculate:
\begin{lemma}
	\label{0 term}
	$H^0(\Omega_X(logD_1, logD_2, logD_3)\otimes\omega_X)=0.$
\end{lemma}

\begin{proof} By Catanese \cite{Cat84}(2.12), we have the following exact sequence
	$$0\rightarrow \Omega_X\otimes \omega_X\rightarrow \Omega_X(logD_1, logD_2, logD_3)\otimes\omega_X\rightarrow \bigoplus_{i=1}^3\mathcal{O}_{D_i}(K_X)\rightarrow 0.$$
	Since $\sigma: X\rightarrow \mathbb{P}^2$ is the blowing up of $\mathbb{P}^2$ at four points, we have  the following exact sequence
	$$0\rightarrow T_X\rightarrow \sigma^*T_{\mathbb{P}^2}\rightarrow \bigoplus_{i=1}^4\mathcal{O}_{E_i}(1)\rightarrow 0.$$
	Since $h^j(\sigma^*T_{\mathbb{P}_2})=h^j(T_{\mathbb{P}_2})$, $h^0(T_{\mathbb{P}_2})= \dim Aut(\mathbb{P}^2)=8$, $h^1(T_{\mathbb{P}_2})=h^2(T_{\mathbb{P}_2})=0$ and $h^0(T_X)=\dim Aut(X)=0$,  we see that
	$H^j(\Omega_X\otimes \omega_X)=H^{2-j}(T_X)=0 (j=0,1,2)$.  Since each $D_i(i=1,2,3)$ is a disjoint union of  rational curves    whose intersection number   with $K_X$  equals -1, -2, or -3, we have $H^0(\mathcal{O}_{D_i}(K_X))=0$, hence  $H^0(\Omega_X(logD_1, logD_2, logD_3)\otimes\omega_X)=0.$
\end{proof}

To compute  $h^0(\Omega_X(logD_i) (K_X+L_i))$ $(i=1,2,3)$, we need the following two lemmas in \cite{BC13}:
\begin{lemma} (\cite{BC13} Lemma 4.3)
	\label{replace lemma}
	Assume that $N$ is a connected component of a smooth divisor $D\subset X$, where $X$ is a smooth projective surface.
	Let $M$ be a divisor on Y. Then $$H^0(\Omega_X(log(D-N))(N+M))=H^0(\Omega_X(log(D))(M))$$
	provided $(K_X+2N+M)N<0$.
\end{lemma}
We shall use Lemma \ref{replace lemma} several times in the case where $N\cong \mathbb{P}^1$ and $N^2<0$.
\begin{lemma} (\cite{BC13} Lemma 7.1 (3))
	\label{push forward lemma}
	Consider a finite set of distinct linear forms
	$$l_\alpha :=y-c_\alpha x, \alpha\in A$$
	vanishing at the origin in $\mathbb{C}^2.$
	Let $p: Z\rightarrow \mathbb{C}^2$ be the blow up of the origin, let $D_\alpha$ be the strict transform of the line $L_{\alpha}:=\{l_\alpha=0\}$, and let $E$ be the exceptional divisor.
	
	Let $\Omega_{\mathbb{C}^2}^1((dlogl_{\alpha})_{\alpha \in A})$ be the sheaf of rational 1-forms generated by $\Omega_{\mathbb{C}^2}^1$ and by the differential forms $dlogl_{\alpha}$ as an $\mathcal{O}_{\mathbb{C}^2}$-module and define similarly $\Omega_Z^1((logD_\alpha)_{\alpha\in A})$. Then:
	
	$p_*\Omega_Z^1((logD_{\alpha})_{\alpha \in A})=\{\eta\in \Omega_{\mathbb{C}^2}^1((dlogl_\alpha)_{\alpha \in A})| \eta=\Sigma _\alpha g_{\alpha}dlogl_\alpha +\omega, \omega \in \Omega_{\mathbb{C}^2}^1, \Sigma_\alpha g_\alpha (0)=0\}$.
	
\end{lemma}

Now we calculate $h^0(\Omega_X(logD_i)(K_X+L_i))$ $(i=1,2,3)$ using a method of Bauer-Catanese (cf. \cite{BC13} Lemmas 4.3, 4.4, 4.5, 4.6, 7.1).
\begin{lemma}
	\label{1 term}
	$H^0(\Omega_X(logD_1)(K_X+L_1))=0$.
\end{lemma}
\begin{proof}    By Lemma \ref{replace lemma}, we have
	\begin{align*}
	& \hspace{5mm} 	H^0(\Omega_X(logD_1)(K_X+L_1))  \\
	&=H^0(\Omega_X(logD_1)(E_4))  \\
	&	=H^0(\Omega_X (log(D_1-L_{34}))(L_{34}+E_4)) ~((K_X+2L_{34}+E_4)L_{34}=-2<0) \\
	&	=H^0(\Omega_X (log(L_{14}+L_{24}))(L-E_3))
	\end{align*}
	By Lemma \ref{push forward lemma}, this  is a subspace $V^1$ of $H^0(\Omega_{\mathbb{P}^2}(logl_{14},logl_{24})(1))$ consisting of sections satisfying several linear conditions.
	Choose a coordinate system $(x_1:x_2:x_3)$ on $\mathbb{P}^2$ such that $P_1=(0:1:0), P_2=(1:0:0), P_3=(1:1:1), P_4=(0:0:1)$. Then $l_{14}=\{x_1=0\}, l_{24}=\{ x_2=0\}$. By \cite{BC13} Lemma 4.5 and Corollary 4.6, any $\omega\in H^0(\Omega_{\mathbb{P}^2}(logl_{14},logl_{24})(1))$ has the form $\omega=\frac{dx_1}{x_1}(a_{12}x_2-a_{21}x_1+a_{13}x_3)+\frac{dx_2}{x_2}(-a_{12}x_2+a_{21}x_1+a_{23}x_3)+dx_3(-a_{13}-a_{23})$ ($a_{ij}\in \mathbb{C}$).

	Now let $\omega\in V^1$. Using  Lemma \ref{push forward lemma} for $P_4$, we get
	$$a_{13}+a_{23}=0.$$
	Using Lemma \ref{push forward lemma} for $P_1, P_2$, we get
	$$a_{12}=a_{21}=0.$$
	Since $\omega(P_3)=a_{13}dx_1+a_{23}dx_2+(-a_{13}-a_{23})dx_3=0$, we get $$a_{13}=a_{23}=0.$$
	Therefore, $H^0(\Omega_X(logD_1)(K_X+L_1))=V^1=0$.
\end{proof}

\begin{lemma}
	\label{3 term}
	$h^0(\Omega_X(logD_3)(K_X+L_3))=1.$
\end{lemma}
\begin{proof} We use the same notation as   in Lemma \ref{1 term}. By Lemma \ref{replace lemma},
	\begin{align*}
	&  \hspace{5mm}  H^0(\Omega_X(logD_3)(K_X+L_3)) \\
	&=H^0(\Omega_X (log(D_3))(L_{12}-E_3)) \\
	&=H^0(\Omega_X(log(C+E_1+E_2))(L_{12}))~((K_X+2E_3+L_{12}-E_3)E_3=-2<0) \\
	&=H^0(\Omega_X (log(C+E_1))(L_{12}+E_2))~((K_X+2E_2+L_{12})E_2=-2<0)) \\
	&=H^0(\Omega_X (log(C))(L)) ~((K_X+2E_1+L_{12}+E_2)E_1=-2<0) \\
	&=H^0(\Omega_X (log(L-E_4))(L))
	\end{align*}
	which is a subspace $V^3$ of $H^0(\Omega_{\mathbb{P}^2} log(l_4)(1))$. Take a coordinate system  $(x_1:x_2:x_3)$ on $\mathbb{P}^2$ such that  $P_4=(0,0,1)$and  $l_4: x_1=0$. By \cite{BC13} Lemma 5,  any element $\omega\in H^0(\Omega_{\mathbb{P}^2} log(l_4)(1))$ has the form $\omega=\frac{dx_1}{x_1}(a_2x_2+a_3x_3)-a_2dx_2-a_3x_3$ ($a_2,a_3\in \mathbb{C}$).  Now let $\omega\in V^3$, using Lemma \ref{push forward lemma} for  $P_4$, we get  $a_3=0$, hence we have
	$V^3\cong \mathbb{C}$. Therefore $h^0(\Omega_X(logD_3)(K_X+L_3))=1.$
\end{proof}

\begin{lemma}
	\label{2 term}
	$h^0(\Omega_X(logD_2)(K_X+L_2))\leq 2.$
\end{lemma}
\begin{proof} By Lemma \ref{replace lemma},
	\begin{align*}
	& \hspace{5mm} H^0(\Omega_X(logD_2)(K_X+L_2)) \\
	&=H^0(\Omega_X(logD_2)(-L_{12}+E_3-E_4)) \\
	&=H^0(\Omega_X(log(D_2-L_{12}))(E_3-E_4))~((K_X+2L_{12}-L_{12}+E_3-E_4)L_{12}=-2<0) \\
	&=H^0(\Omega_X(log(D_2-L_{12}-E_4))(E_3))~((K_X+2E_4+E_3-E_4)E_4=-2<0) \\
	&=H^0(\Omega_X(log(Q+L_{23}+L_{13})(E_3)) \\
	&=H^0(\Omega_X(log(Q+L_{23})(L_{13}+E_3))~((K_X+2L_{13}+E_3)L_{13}=-2<0) \\
	&=H^0(\Omega_X(log(Q+L_{23})(L-E_1))
	\end{align*}
	which is  a subspace $V^2$ of $H^0(\Omega_{\mathbb{P}^2}( log(\bar{Q} + l_{23}))(1))$.   From the exact sequence
	$$0\rightarrow \Omega_{\mathbb{P}^2}(1)\rightarrow \Omega_{\mathbb{P}^2}(log(\bar{Q} + l_{23}))(1)\rightarrow \mathcal{O}_{\bar{Q}+l_{23}}(1)\rightarrow 0,$$
	and  $h^i(\Omega_{\mathbb{P}^2}(1))=0$ $(i=0,1)$, we get
	$$H^0(\Omega_{\mathbb{P}^2} (log(\bar{Q} + l_{23}))(1)) \cong H^0(\mathcal{O}_{\bar{Q}+l_{23}}(1)),$$
	which has dimension 3.
	
	Choose a coordinate system $(x_1:x_2:x_3)$ on $\mathbb{P}^2$ such that $P_1=(1:0:0), P_2=(0:1:0), P_3=(0:0:1)$. Then $l_{23}=\{x_1=0\}, \bar{Q}=\{x_1x_2+x_2x_3+x_1x_3=0\}$. Note that any element  $\omega\in H^0(\Omega_{\mathbb{P}^2} log(\bar{Q} + l_{23})(1))$ is of the form $\omega=(a_1x_1+a_2x_2+a_3x_3)(\frac{dx_1}{x_1} +        \frac{d(x_1x_2+x_2x_3+x_1x_3)}{x_1x_2+x_2x_3+x_1x_3})$ ($a_1,a_2,a_3\in \mathbb{C}$). Now let $\omega\in V^2$, since  $\omega(P_1)=0$, we get $a_1=0$. Therefore $h^0(\Omega_X(logD_2)(K_X+L_2))=$dim$V^2\leq 2$.
\end{proof}

On the other hand,
By \cite{Cat84} (2.18), we have the following exact sequence (since $S$ is of general type, we have $H^0(T_S)=0$)

$$0=H^0(T_S)\rightarrow H^0(\pi^*T_X) \rightarrow \bigoplus _{i=1}^3 H^0(\mathcal{O}_{D_i}(D_i)\oplus \mathcal{O}_{D_i}(D_i-L_i)) \xrightarrow{\partial} H^1(T_S) $$$$\rightarrow H^1(\pi^*T_X)\rightarrow \bigoplus_{i=1}^3 H^1(\mathcal{O}_{D_i}(D_i)\oplus \mathcal{O}_{D_i}(D_i-L_i)) \rightarrow H^2(T_S)\rightarrow H^2(\pi^*T_X) \rightarrow 0.$$
\begin{lemma}
	\label{cohomology of tangent sequence}
	$H^0(\pi^*T_X)=H^0(T_X)\oplus H^0(T_X(-L_1))\oplus H^0(T_X(-L_2))\oplus H^0(T_X(-L_3))=0$, so the map $$\partial:  \bigoplus_{i=1}^3 H^0(\mathcal{O}_{D_i}(D_i)\oplus \mathcal{O}_{D_i}(D_i-L_i)) \rightarrow H^1(T_S) $$
	is injective.
	Moreover, we have
	$$h^0(\mathcal{O}_{D_1}(D_1) \oplus \mathcal{O}_{D_1}(D_1-L_1))=h^0(\mathcal{O}_{D_1}(D_1))=0;$$
	$$h^0(\mathcal{O}_{D_2}(D_2)\oplus \mathcal{O}_{D_2}(D_2-L_2))=h^0(\mathcal{O}_{D_2}(D_2))=2;$$
	$$h^0(\mathcal{O}_{D_3}(D_3)\oplus \mathcal{O}_{D_3}(D_3-L_3))=h^0(\mathcal{O}_{D_3}(D_3))=1.$$
	Therefore,  $$h^1(T_S)\geq \Sigma_{i=1}^3h^0(\mathcal{O}_{D_i}(D_i)\oplus \mathcal{O}_{D_i}(D_i-L_i))=3.$$
\end{lemma}

\begin{proof}  By the projection formula, we have $\pi_*(\pi^*T_X)=T_X\oplus(\bigoplus_{i=1}^3  T_X(-L_i))$. Since $\pi$ is an affine morphism, we have  $H^0(\pi^*T_X)= H^0(\pi_*\pi^*T_X)=H^0(T_X\oplus(\bigoplus_{i=1}^3  T_X(-L_i)))$. Since $h^0(T_X)=0$ and $L_i $ $(i=1,2,3)$  are    effective divisors, we see $h^0(T_X(-L_i))\leq h^0(T_X)=0$.   Hence $\partial$ is injective and $h^1(T_S)\geq \sum_{i=1}^3h^0(\mathcal{O}_{D_i}(D_i)\oplus \mathcal{O}_{D_i}(D_i-L_i))$.
	
	Note each $D_i$ is a disjoint union of smooth rational curves. Now  the lemma follows  from
	
	$\mathcal{O}_{D_1}(D_1)\cong \mathcal{O}_{L_{14}}(-1)\oplus \mathcal{O}_{L_{24}}(-1)\oplus \mathcal{O}_{L_{34}}(-1);$
	
	$\mathcal{O}_{D_1}(D_1-L_1)\cong \mathcal{O}_{L_{14}}(-3)\oplus \mathcal{O}_{L_{24}}(-3)\oplus\mathcal{O}_{L_{34}}(-2);$
	
	$\mathcal{O}_{D_2}(D_2)\cong \mathcal{O}_{L_{12}}(-1)\oplus \mathcal{O}_{L_{13}}(-1)\oplus\mathcal{O}_{L_{23}}(-1) \mathcal{O}_{E_4}(-1) \oplus \mathcal{O}_Q(1);$
	
	$\mathcal{O}_{D_2}(D_2-L_2)\cong \mathcal{O}_{L_{12}}(-3)\oplus \mathcal{O}_{L_{13}}(-2)\oplus \mathcal{O}_{L_{23}}(-3) \mathcal{O}_{E_4}(-2) \oplus\mathcal{O}_Q(-3);$
	
	$\mathcal{O}_{D_3}(D_3)\cong \mathcal{O}_{E_1}(-1)\oplus \mathcal{O}_{E_2}(-1)\oplus\mathcal{O}_{E_3}(-1) \oplus \mathcal{O}_C;$
	
	$\mathcal{O}_{D_3}(D_3-L_3)\cong \mathcal{O}_{E_1}(-3)\oplus \mathcal{O}_{E_2}(-3)\oplus \mathcal{O}_{E_3}(-3) \oplus \mathcal{O}_C(-3).$
\end{proof}

\begin{pro}
	\label{irreducible component}
	Let $S$ be a  general surface in $M$. Then $h^1(T_S)=3$ and  $\overline{\mathcal{M}}$ is an irreducible component of $\mathcal{M}^{5,2}_{1,1}$. Moreover,  every small deformation of $S$ is a natural deformation (cf.   \cite{Cat84} Definition2.8).
\end{pro}
\begin{proof} By Lemmas  \ref{0 term}-\ref{2 term},  we have  $h^1(T_S) \leq 3$; by Lemma \ref{cohomology of tangent sequence}, we have $h^1(T_S) \geq 3$, thus we get $h^1(T_S)=3=\dim {\mathcal{M}}$.  Hence $\overline{\mathcal{M}}$ is an irreducible component of $\mathcal{M}^{5,2}_{1,1}$.

	Since $h^1(T_S)=3$, the map $$\partial:  \bigoplus_{i=1}^3 H^0(\mathcal{O}_{D_i}(D_i)\oplus \mathcal{O}_{D_i}(D_i-L_i)) \rightarrow H^1(T_S) $$ in Lemma \ref{cohomology of tangent sequence}
	is bijective.
	
	Since the natural  restriction map $r_i: H^0(\mathcal{O}_X(D_i)\oplus \mathcal{O}_X(D_i-L_i)) \rightarrow H^0(\mathcal{O}_{D_i}(D_i)\oplus \mathcal{O}_{D_i}(D_i-L_i))$ is surjective for each $i$,  the composition map
	$$\rho:=\partial\circ(\sum_{i=1}^3r_i): \bigoplus_{i=1}^3 H^0(\mathcal{O}_X(D_i)\oplus \mathcal{O}_X(D_i-L_i))\rightarrow H^1(T_S)$$
	is also surjective. Therefore every small deformation of $S$ is a natural deformation.
\end{proof}

\section{Comparison  with  Catanese-Pignatelli's   structure theorem for genus 2 fibrations }
In this section, we     study  the 5-tuple $(B, V_1, \tau, \xi, \omega)$  in Catanese-Pignatelli's structure theorem for genus 2 fibrations (see \cite{CP06} section 4)  for     the Albanese fibration of a surface  $S\in M$.  We first recall some basic definitions related to   Catanese-Pignatelli's structure theorem for genus 2 fibrations.

Given  a relatively minimal genus 2 fibration  $f:S\rightarrow B$,
set $V_n:=f_*\omega_{S/B}^{\otimes n}$ and $\mathcal{R}:=\oplus_{n=0}^{\infty}V_n$. Let   $S\rightarrow S'$ be the contraction of $(-2)$ -curves. Note that the genus 2 fibration $f$ induces an involution $j'$ on $S$, which maps $(-2)$-curves to $(-2)$-curves and thus induces an involution $j$ on $S'$. Hence  We have  a conic bundle   $\pi_\mathcal{C}:\mathcal{C}:=S'/j\rightarrow B$.
The  5-tuple $(B,V_1,\tau, \xi, \omega)$ associated to $f$  is defined as follows:

  $B$ is the base curve;

  $V_1=f_*\omega_{S/B}$;

  $\tau$ is the effective divisor on $B$ whose structure sheaf is isomorphic to the cokernel of the morphism $S^2(V_1)\rightarrow V_2$ (induced by multiplication in $\mathcal{R}$), whose support corresponds to the singular fibres of the conic bundle $\mathcal{C}\rightarrow B$;

  $\xi \in \Ext^1_{\mathcal{O}_B}(\mathcal{O}_\tau, S^2(V_1))/ Aut_{\mathcal{O}_B}(\mathcal{O}_\tau)$ corresponds  to the  extension
  $$0\rightarrow S^2(V_1)\stackrel{\upsilon}{\rightarrow} V_2 \rightarrow \mathcal{O}_\tau\rightarrow 0;$$

  $\omega\in\mathbb{P}(H^0(B,\mathcal{A}_6\otimes (\det V_1\otimes \mathcal{O}_B(\tau))^{-2}))\cong  |\mathcal{O}_{\mathcal{C}}(6)\otimes \pi_\mathcal{\mathcal{C}}^*(\det V_1\otimes \mathcal{O}_B(\tau))^{-2}|$ corresponds to the branch divisor of the double cover  $S'\rightarrow \mathcal{C}$,
where $\mathcal{A}_6$ is defined as follows.
Consider the map
$i_n:  (\det V_1)^2\otimes S^{n-2}(V_2)\rightarrow S^n(V_2)$ ($n\geq 2$) defined locally by $i_n((x_0\wedge x_1)^{\otimes2}\otimes q)=(\upsilon(x_0^2)\upsilon(x_1^2)-\upsilon(x_0x_1)^2)q$, where where $x_0,x_1$ are generators of the stalk of $V_1$ and $q$ is an element of the stalk of $S^{n-2}(V_2)$ at a point. Define $\mathcal{A}_{2n}$ to be the cokernel of $i_n$. In particular $\mathcal{A}_6$ is the cokernel of $i_3$.

We remark that $\mathcal{C}=\mathbf{Proj}(\mathcal{A})$, where $\mathcal{A}$ is a graded $\mathcal{O}_B$ module defined by the 5-tuple  $(B,V_1,\tau, \xi, \omega)$ as follows.
Consider the map $j_n: V_1\otimes (\det V_1)\otimes \mathcal{A}_{2n-2}\rightarrow V_1\otimes \mathcal{A}_{2n}$ ($n\geq 1$) locally defined by $j_n(l\otimes (x_0\wedge x_1)\otimes q)=x_0\otimes (\upsilon(x_1l)q)-x_1\otimes (\upsilon(x_0l)q)$, where $x_0,x_1,q$ are as before and  $l$ is an element of the stalk of $V_1$ at a point. Define $\mathcal{A}_{2n+1}$ ($n\geq 1$) to be the cokernel of $j_n$. By \cite{CP06} Lemma 4.4, $\mathcal{A}_n$ is a locally free sheaf  on $B$ for all $n\geq 3$. Let $\mathcal{A}_0:=\mathcal{O}_B$,  $\mathcal{A}_1:=V_1$ and $\mathcal{A}_2:=V_2$.  Then $\mathcal{A}:= \oplus_{n\geq 0}\mathcal{A}_n$.

\vspace{3ex}
Now   we prove that surfaces in $M$ are in one to one correspondence with minimal surfaces    satisfying the following condition:

$(\star)$  \quad   $p_g=q=1, K^2=5, g=2$; after choosing an appropriate neutral  element $0$ for the genus one curve  $B=Alb(S)$, $V_1=E_{[0]}(2,1), V_2=\mathcal{O}_B(2\cdot0) \oplus \mathcal{O}_B(2\cdot0) \oplus \mathcal{O}_B(2\cdot0)$ and $\tau=\eta_1+\eta_2+\eta_3$, where $\eta_1, \eta_2, \eta_3$ are the three nontrivial  2-torsion points on $B$.

\vspace{3ex}
First we show that a general surface $S\in M$ satisfies  condition $(\star)$.
\begin{lemma}
	\label{compute the 5-tuple}
	Let $S$ be a general surface in $M$. Then $S$  satisfies  condition ($\star$).
\end{lemma}
\begin{proof}  	
	
	We use notation of section 3.  The bidouble cover $\pi: S\rightarrow X$ can be regarded as two successive  double covers $\pi_1: \mathcal{C}\rightarrow X$   branched  over  $D_1\cup D_3$ and $\pi_2: S\rightarrow \mathcal{C}$  branched  over  $\pi_1^*(D_2\cup (D_1\cap D_3))$. Note that  $D_1\cup D_3$ is the union of a smooth fibre (over $\gamma_0\in \mathbb{P}^1$) and three singular fibres (over $\gamma_i\in \mathbb{P}^1$ $(i=1,2,3)$) of the natural fibration  $g: X\rightarrow \mathbb{P}^1$. Let $\mu: B'\rightarrow \mathbb{P}^1$ be the double cover with branch divisor $\gamma_0+ \gamma_1+ \gamma_2+ \gamma_3$. Then $g(B')=1$. Moreover, there is   a unique  (singular)  fibration $\tilde{g}: \mathcal{C}\rightarrow B'$ such that the  following   diagram
	$$\xymatrix{\mathcal{C}\ar[r]^{\pi_1} \ar[d]^{\tilde{g}} &X\ar[d]^g\\
		B'\ar[r]^\mu &\mathbb{P}^1
	}$$
	commutes. Since the  general fibre $F$  of $f:=\pi_2\circ \tilde{g}: S\rightarrow B'$ is connected, by the universal property of Albanese map, we know that $B'=B$. Moreover, $\mathcal{C}$ is exactly the conic bundle in Catanese-Pignatelli's structure theorem for genus 2 fibrations. Fix a group law for $B$ and let  $\mu^{-1}(\gamma_0)$ be the neutral element $0\in B$.  Then $\eta_i:=\mu^{-1}(\gamma_i)$ $(i=1,2,3)$ are the three nontrivial 2-torsion points on $B$.  Since  $\mathcal{C}$ has exactly three nodes on the three fibres over $\eta_1, \eta_2, \eta_3$, we know $\tau=\eta_1+\eta_2+\eta_3$.
	
	Since $h^0(V_2(-2\cdot 0))=h^0(2K_S-2F_0)=h^0(\pi^*(2K_X+D-\Gamma))=3$ (where $\Gamma$ is a  fibre of $g$), by a similar argument as in Lemma \ref{limit V1 V2}, one can show easily that  $V_2=\mathcal{O}_B(2\cdot0) \oplus \mathcal{O}_B(2\cdot0) \oplus \mathcal{O}_B(2\cdot0)$.
	
	Now we show that $V_1=E_{[0]}(2,1)$. Let $\tilde{\mathcal{C}}\rightarrow \mathcal{C}$ be the minimal resolution of $\mathcal{C}$, then the pull back of each singular fibre of $\mathcal{C}$ is a union of a $(-2)$ curve and two $(-1)$ curves. Contracting the six $(-1)$ curves of $\tilde{\mathcal{C}}$, we get a smooth ruled surface, which is exactly the second symmetric product $B^{(2)}$ of $B$.  Let $\lambda: \mathcal{C}\dashrightarrow B^{(2)}$ be the birational map above. Then   we get a rational double cover $\pi_2':=\lambda\circ\pi_2: S\dashrightarrow B^{(2)}$.
	
	Let $p: B^{(2)}\rightarrow B$ be the natural projection,
	let  $D_u:=\{(x,u)|x\in B\}$ be a section  and  $E_v:=\{(x,v-x)| x\in B\}$ be a fibre of $p$ (cf. \cite{CP06} p. 1028).  Then the branch divisor of $\pi_2'$ consists of:   three fibres  $E_{\eta_1}, E_{\eta_2}, E_{\eta_3}$;  four sections $D_0, D_{\eta_1}, D_{\eta_2}, D_{\eta_3}$;  and a bisection $\equiv 2D_0+E_0$ passing though    $Q_i:=(0,\eta_i)(i=1,2,3)$,$Q_4:=(\eta_1,\eta_2)$,$Q_5:=(\eta_1,\eta_3)$ and $Q_6:=(\eta_2,\eta_3)$.
	Hence  we have
	$$|K_S|\cong  |\pi_2'^*(D_0+3E_0-\Sigma_{i=1}^6 Q_i)| \cong  |D_0|+|E_{\eta_1}+E_{\eta_2}+E_{\eta_3}-\Sigma_{i=1}^6 Q_i)|.$$
	It is easy to see that  $V_1=f_*\omega_S=(\tilde{g}\circ \pi_2)_*\omega_S=p_*\mathcal{O}_{B^{(2)}}(D_0)=E_{[0]}(2,1)$.
\end{proof}

Next we show that any surface $S\in M$  satisfies  condition $(\star)$.

\begin{lemma}
	\label{limit V1 V2}
	Let $p: \mathcal{S} \rightarrow T$ be a 1-parameter   family of minimal surfaces  with base $T\ni 0$ connected and smooth. Assume that  for any $0\neq t\in T$, $\mathcal{S}_t$    satisfies the condition $(\star)$. Then $\mathcal{S}_0$ also satisfies   condition $(\star)$.
\end{lemma}
\begin{proof} Note that  $p_g, q, K^2$, the number of the direct summands of $V_1$ (cf. Remark \ref{deformation invariant} below)  and  the genus $g$ of Albanese fibre   (cf. \cite{CP06} Remark 1.1)  are all differentiable invariants, hence they are also deformation invariants. Therefore,    $\mathcal{S}_0$ also has  $p_g=q=1, K^2=5, g=2$   such that  $V_1$ is an indecomposable rank 2 vector bundle of degree 1.

	Taking a  base change and replacing $T$ with  a (Zariski) open subset if necessary, we can assume that $p$ has a section $s: T \rightarrow \mathcal{S}$, so  we can choose base points $x_0$ for all $\mathcal{S}_t:=p^{-1}(t)$ $ (t\in T)$ simultaneously, therefore we can define the Albanese map ($x\mapsto \int_{x_0} ^x$) for all $\mathcal{S}_t$ $(t\in T)$ simultaneously. Thus we get a smooth family $q: \mathcal{B} \rightarrow T$ with $\mathcal{B}_t:=q^{-1}(t)=Alb(\mathcal{S}_t)$ $(t\in T)$, which also has a section induced by $s$. Hence  we can choose the neutral element 0 for all $\mathcal{B}_t$ $(t\in T)$ simultaneously and  assume $V_1=E_{[0]}(2,1)$ for $\mathcal{S}_0$. Moreover we have the following commutative diagram:
	
	$$
	\xymatrix{
		\mathcal{S} \ar[r]^\alpha \ar[rd]^p & \mathcal{B} \ar[d]^q
		\\
		~ & T}
	$$

	Now we use the upper semi-continuity for $h^0(V_2(-2\cdot0))$. Since for $\mathcal{S}_t$ ($t\neq 0$), we have $h^0(V_2(-2\cdot0))=3$,  we  have $h^0(V_2(-2\cdot0))\geq3$ for $\mathcal{S}_0$.    Set $B:=\mathcal{B}_0$.
	
	(i) If $V_2$ is indecomposable, then $V_2=F_2(2b)$ for some point $b\in B$ (here $F_2$ is the unique indecomposable rank 2 vector bundle over $B$ with $\det F_2=\mathcal{O}_B$), so $h^0(V_2(-2\cdot0))\leq 1$, a contradiction;
	
	(ii) If $V_2=W\oplus L$ for some rank 2 indecomposable vector bundle $W$ and some line bundle $L$, then by the exact sequence
	$$0\rightarrow \bigoplus_{i=1}^3\mathcal{O}_B(\eta_i) \rightarrow V_2 \rightarrow \mathcal{O}_{\tau} \rightarrow 0$$
	we know that  $\deg W\geq 2, \deg L\geq 1$. Since $\deg W+\deg L=6$,  we know  $(\deg W, \deg L)= (2,4), (3,3), (4,2)$ or $(5,1)$. In  all cases above, we always have $h^0(V_2(-2\cdot0))\leq 2$, a contradiction;
	
	(iii) If $V_2$ is a direct sum of three line bundles $L_1, L_2, L_3$,   w.l.o.g. we can assume $\deg L_1\leq \deg L_2\leq \deg L_3$. From the exact sequence
	$$0\rightarrow \bigoplus_{i=1}^3\mathcal{O}_B(\eta_i) \rightarrow V_2 \rightarrow \mathcal{O}_{\tau} \rightarrow 0$$
	we get $\deg L_i\geq 1$ $(i=1,2,3)$, thus  $(\deg L_1, \deg L_2, \deg L_3)=(1,1,4)$,  $(1,2,3)$ or $(2,2,2)$. In the first two cases,  we have   $h^0(V_2(-2\cdot0))\leq 2$, a contradiction; in the last case, we see that   $h^0(V_2(-2\cdot0))\geq3$ if and only if $L_i\cong \mathcal{O}_B(2\cdot0)$ for  all $i$.
	Hence  for $\mathcal{S}_0$, we also have $V_2= \mathcal{O}_B (2\cdot0) \oplus \mathcal{O}_B(2\cdot0) \oplus \mathcal{O}_B(2\cdot0)$.
	
	By the following Remark \ref{deformation invariant} (ii), we see  $\tau=\eta_1+\eta_2+\eta_3$ for $\mathcal{S}_0$.  Therefore $\mathcal{S}_0$ also satisfies  condition $(\star)$.
\end{proof}

\begin{remark}
	\label{deformation invariant}
	(i)  Catanese-Ciliberto (\cite{CC91} Theorem 1.4, Proposition 2.2) proved that  the number of the direct summands of $V_1$ is a topological invariant; however,  the case of $V_2$ is quite different,  as we shall show in section 6  that  the number of the direct summands of $V_2$ is even not a deformation invariant.

	(ii)	If S is a  surface with $p_g=q=1, K^2=5, g=2$ such that  $V_1=E_{[0]}(2,1)$, $V_2=\mathcal{O}_B(2\cdot0) \oplus \mathcal{O}_B(2\cdot0) \oplus \mathcal{O}_B(2\cdot0)$, then   we can
	choose a suitable  coordinate system  $(y_1:y_2:y_3)$ on  the fibre   of $\mathbb{P}(V_2)=B\times \mathbb{P}^2\rightarrow B$ such that  the matrix of the map $\sigma_2: S^2(V_1)\rightarrow V_2$ is diagonal (see \cite{Pig09} Proposition 4.5),  then  $\tau=\eta_1+\eta_2 +\eta_3$ (where $\tau$ is one of the 5-tuple in Catanese-Pignatelli's structure theorem for genus 2 fibrations) and the equation of the conic bundle  $\mathcal{C} \subset \mathbb{P}(V_2)$ is  $a_1^2y_1^2 +a_2^2y_2^2+a_3^2y_3^2=0$ (here $a_i\in H^0(\mathcal{O}_B(\eta_i))$).
	In particular, $\mathcal{C}$ has exactly   three nodes $\{a_1=y_2=y_3=0\}$, $\{a_2=y_1=y_3=0\}$, $\{a_3=y_1=y_2=0\}$ on three singular fibres over $\eta_1, \eta_2, \eta_3$.
	
	(iii) By (ii) above,    it is possible that two minimal surfaces $S_1\cong S_2$, but $S_1$ and $S_2$ have different $\tau$. Hence it is possible that  the surfaces with $p_g=q=1,K^2=5,g=2$ constructed by Ishida \cite{Ish05} are isomorphic to some surfaces in $M$.
\end{remark}

Now we can  prove the following:
\begin{pro}
	\label{5-tuple for all surface in M}
	Let $S$ be any  surface in  $M$. Then $S$ satisfies      condition $(\star)$.
\end{pro}
\begin{proof}  Let $S$ be a surface in $M$ and let $S'$ be its canonical model. If $S'=S$,  then $S$ is a smooth bidouble cover of $X$. By Lemma \ref{compute the 5-tuple}, $S$ satisfies  condition $(\star)$.
	
	If $S'$ is singular, since a general surface in $M$ has smooth canonical model, we can find a smooth  1-parameter  family $p: \mathcal{S} \rightarrow T$ such that  $\mathcal{S}_0=S$ and $\mathcal{S}_t$ $(t\neq 0)$ is a general surface in $M$. By Lemma \ref{limit V1 V2}, $S=\mathcal{S}_0$ also  satisfies   condition $(\star)$.
\end{proof}

In the following, we show that the converse of Proposition \ref{5-tuple for all surface in M} is also true.
\begin{lemma}
	\label{canonical model is a bidouble cover}	
	Let $S$ be a  minimal surface  satisfying   condition $(\star)$. Then the canonical model $S'$ of $S$ is a bidouble cover of $X$.
\end{lemma}
\begin{proof}
	The Albanese fibration of $S$ induces an involution $i'$ on $S$, which maps $(-2)$ curves to $(-2)$ curves, thus induces an involution $i$ on the canonical model $S'$ of $S$. The quotient $\mathcal{C}:=S'/i$ is nothing but the conic bundle in Catanese-Pignatelli's  structure theorem. By Remark \ref{deformation invariant}, after choosing a suitable coordinate system  $(y_1:y_2:y_3)$ on  the fibre  of $\mathbb{P}(V_2)=B\times \mathbb{P}^2\rightarrow B$, we can assume that  the equation of the conic bundle  $\mathcal{C} \subset \mathbb{P}(V_2)$ is  $$a_1^2y_1^2 +a_2^2y_2^2+a_3^2y_3^2=0$$ (here $a_i\in H^0(\mathcal{O}_B(\eta_i))$).
	
	There is an  involution $j'$ on $B\times \mathbb{P}^2$ induced by the involution $j_o: u\mapsto -u$ on $B$.  Since $\mathcal{C}$ is invariant under $j'$, $j'$ induces an  involution $j$ on $\mathcal{C}$.  If we denote by $\iota: B\rightarrow \mathbb{P}^1$ the quotient map induced by $j_o$ and denote  by $(x_1:x_2:x_3)$ the coordinate system  on  the fibre   of $\mathbb{P}^1\times \mathbb{P}^2\rightarrow \mathbb{P}^1$ corresponding to $(y_1:y_2:y_3)$. Then the equation of $X':=\mathcal{C}/j $ is
	$$h:=b_1x_1^2 +b_2x_2^2+b_3x_3^2=0$$
	where $b_i\in H^0(\mathcal{O}_{\mathbb{P}^1}(\iota(\eta_i)))$. Since  the Jacobian matrix of $h$ always has rank 1, $X'$ is a smooth surface  of bi-degree $(1,2)$ in $\mathbb{P}^1\times \mathbb{P}^2$. In particular,  $-K_{X'}$ is ample and $K_{X'}^2=5$, which implies that  $X'$ is  the  Del Pezzo surface  $X$ of degree 5.
	
	Now  we have  two successive double covers $\pi_1: \mathcal{C}\rightarrow X$ and $\pi_2: S'\rightarrow \mathcal{C}$.  We only need to show that   the composition $\pi:=\pi_1\circ \pi_2: S'\rightarrow X$ is really a bidouble cover.

	Let $p_1,p_2$ be  the natural projection from $\mathbb{P}^1 \times \mathbb{P}^2$ to $\mathbb{P}^1, \mathbb{P}^2$ respectively and let $T:= p_2^*\mathcal{O}_{\mathbb{P}^2}(1),F:=p_1^*\mathcal{O}_{\mathbb{P}^1}(1)$; let $\tilde{p}_1,\tilde{p}_2$ be  the natural projection from $B \times \mathbb{P}^2$ to $B, \mathbb{P}^2$ respectively and let $\tilde{T}:= \tilde{p}_2^*\mathcal{O}_{\mathbb{P}^2}(1),\tilde{F}:=\tilde{p}_1^*\mathcal{O}_{\mathbb{P}^1}(1)$. Denote by $\Delta_1,\Delta_2$ the branch divisor of $\pi_1,\pi_2$ respectively, then $\Delta_1\equiv (4F)|_X, \Delta_2\equiv (3\tilde{T}-2\tilde{F}_0)|_{\mathcal{C}}$. To show that $\pi:=\pi_1\circ \pi_2$ is a bidouble cover, it suffices to show that $\Delta_2$ is invariant under $j$: if so, we can lift $j$ to an involution $\tilde{j}$ on $S'$, hence we get a group $G:=\{1, i, \tilde{j}, i\circ \tilde{j}\}\cong (\mathbb{Z}/2\mathbb{Z})^2$ acting on $S'$ and the quotient $S'/G$ is nothing but $X$. Therefore $\pi: S'\rightarrow X$ is a bidouble cover.

	\vspace{3ex}
	Now we show that $\Delta_2$ is invariant under $j$. To show this, it suffices to show that  $\Delta_2=\pi^*D$ for some effective  divisor $D$ on $X$. Since $\Delta_2\equiv \pi_1^*(3T+F)|_{\mathcal{C}}$, it suffices to show   $H^0(\Delta_2)\cong H^0((3T+F)|_X)$.
	Since  $H^0(\Delta_2)\cong H^0((3T+F)|_X)\oplus H^0((3T-F)|_X)$,   we only need to show $H^0((3T-F)|_X)=0$.
	
	Using   the same notation $L, E_i$ of section 3.2, up to an automorphism of $X$, we have $T|_X\equiv 2L-E_1-E_2-E_3, F|_X\equiv L-E_4$. If $H^0((3T-F)|_X)\neq 0$, then there is an effective divisor $D'\equiv (3T-F)|_X\equiv 3L-3E_1-3E_2-3E_3+3E_4$. Since $D'E_4<0$, $(D'-E_4)E_4<0$ and $(D'-2E_4)E_4<0$,   $3E_4$ is contained in the fixed part of $D'$. Thus $D'':=D'-3E_4$ is also an effective divisor. Since $-K_X\equiv 3L-E_1-E_2-E_3-E_4$ is ample and $(-K_X)D''=0$, we get $D''= 0$, a contradiction. Hence $H^0((3T-F)|_X)=0$.
	
	Therefore $\Delta_2$ is invariant under $j$ and consequently $\pi: S'\rightarrow X$ is a bidouble cover.
\end{proof}

\begin{pro}
	\label{inverse of 5-tuple}
	Let $S$ be a minimal  surface  satisfying   condition $(\star)$.  Then $S\in M$.
\end{pro}
\begin{proof} By Lemma \ref{canonical model is a bidouble cover}, we only need to prove that the effective divisors $(D_1,D_2,D_3)$ and divisors $(L_1,L_2,L_3)$ of the bidouble cover $\pi: S'\rightarrow X$  are of the same form as in   section 3.2.
	
	We use the    notation  $L, E_i$,  $L_{ij}$, $Q$, $C$  of section 3.
	By  Lemma \ref{canonical model is a bidouble cover}, if we denote by $R_2$ the fixed part  of the involution $i$ on $S'$, then $D_2=(\pi_2)_* R_2\equiv 5L-3E_1-3E_2-3E_3+E_4$. Since dim$|D_2|=2$ and $|D_2|$ contains a 2-dimensional sub-linear system  $\mathcal{L}$ of divisors of the form $Q+L_{12}+L_{23}+L_{13}+E_{4}$, we see $|D_2|=\mathcal{L}$. So $D_2$ must be of the form $Q+L_{12}+L_{23}+L_{13}+E_{4}$. Since $D_2$ is reduced, $Q$ must be the strict transform of a smooth conic, thus  $D_2$ is always smooth.
	
	Since the branch divisor  of the bidouble cover $\pi: S'\rightarrow X$ is $D_2\cup \Delta_1$, we get  $D_1\cup D_3=\Delta_1\equiv 4L-4E_4$. Since  $D_1+D_2$ and $D_3+D_2$ are both  effective even divisors (cf. \cite{Cat84} (2.1)),  we can assume   $D_1\equiv 3L + \Sigma a_{1i}E_i, D_3\equiv L+\Sigma a_{3i}E_i$, where $a_{1i}, a_{3i}$ are   odd  integers for $i=1,2,3,4$, and $a_{14}+a_{34}=-4$. Since $D=D_1\cup D_2\cup D_3$ is reduced, one can easily show  that up to an automorphism of $X$,  $D_1=L_{14}+L_{24}+L_{34}\equiv 3L-E_1-E_2-E_3-3E_4$; $D_3=C+E_{1}+E_{2}+E_{3}\equiv L+E_1+E_2+E_3-E_4$, which are  the same as in section 3.
	
	Since $Pic(X)$ has no nontrivial 2-torsion elements, $(L_1,L_2,L_3)$ are uniquely determined by $(D_1,D_2,D_3)$ through the linear equivalence relations $2L_i\equiv D_j+D_k (\{i,j,k\}=\{1,2,3\})$.  Therefore $S\in M$.
\end{proof}

Combining Propositions \ref{5-tuple for all surface in M} and   \ref{inverse of 5-tuple}, we get  the  following theorem, which  plays a crucial role in proving  the  (Zariski) closedness of $\mathcal{M}$.
\begin{theorem}
	\label{one to one correspondence}
	Every surface $S\in M$  satisfies  condition $(\star)$. Conversely, If $S$ is a minimal surface  satisfying  condition $(\star)$, then $S\in M$.
\end{theorem}

\begin{pro}
	\label{closed}
	$\mathcal{M}$ is a Zariski closed subset of $\mathcal{M}^{5,2}_{1,1}$, i.e. $\mathcal{M}=\overline{\mathcal{M}}$.
\end{pro}
\begin{proof}	It suffices to show:  if $\mathcal{S}'\rightarrow T$ is  a 1-parameter connected flat family  of canonical  models of algebraic surfaces with   $\mathcal{S}'_t$ ($0\neq t\in T$)  a general surface in $\mathcal{M}$, then $\mathcal{S}'_0\in \mathcal{M}$.
	
	Taking a base change (we still denote by $T$ the base curve) and the    simultaneous resolution, we get a connected smooth family  $\mathcal{S}\rightarrow T$ with  $\mathcal{S}_t$ $(t\in T)$    the minimal model of $S'_t$.  Note that  for each $0\neq t\in T$, $\mathcal{S}_t$ is  a general surface in $M$. By Lemma \ref{limit V1 V2} and Theorem \ref{one to one correspondence}, we see that $\mathcal{S}_0\in M$, hence $\mathcal{S}'_0\in \mathcal{M}$.
\end{proof}

At the end of this section, we give some remarks on the branch  curve of the bidouble cover $\pi: S'\rightarrow X$, which we shall use in   the next section.
\begin{remark}
	\label{remark for branch divisors}
	(1) The choice of $(D_1,D_2,D_3)$ in Lemma \ref{inverse of 5-tuple} is not unique (e.g. there are two choices in section 2), but all choices are equivalent up to an automorphism of $X$.
	
	(2) From  Lemma \ref{inverse of 5-tuple}, we see that  each $D_i$ $(i=1,2,3)$ is   smooth. In fact, the only possible singularity on the branch divisor $D=D_1\cup D_2\cup D_3$ is a  node  coming from $B'\cap C$ (here we use notation  in section 2):

	Since   $K_S^2=5$, the conic $B' \subset \mathbb{P}^2$ cannot have the same tangent direction with $A_i$  at  $P_i$ for any $i\in \{1,2,3\}$.  Otherwise  we would finally get a minimal   surface with $K_S^2<5$. So the only possible   singularity  comes from $B'\cap C$ when   $B'$ has the same tangent direction with  $C$ at $B'\cap C$.  This is a node on $S'$.
	
	When  $S'$ is singular, $\mathcal{C}$ is the same as before since $\Delta_1=D_1\cup D_3$ is the same. In particular, $\mathcal{C}$  has exactly three singular fibres. Moreover,
	the branch curve of the  double cover $\pi_2: S'\rightarrow \mathcal{C}$  still has  5 irreducible and connected components: four smooth sections  and a singular curve that is algebraically equivalent to a bisection.
\end{remark}

\section{$\mathcal{M}$ is a connected   component of $\mathcal{M}^{5,2}_{1,1}$}
In this section, we study the deformation of the branch curve of  the  double cover $\pi_2: S'\rightarrow \mathcal{C}\subset \mathbb{P}(V_2)$ (where $\mathcal{C}$ is the conic bundle in Catanese-Pignatelli's structure theorem for genus 2 fibrations, see Lemma \ref{canonical model is a bidouble cover}) and
prove that $\mathcal{M}$ is an analytic open subset of $\mathcal{M}^{5,2}_{1,1}$, i.e.
\begin{pro}
	\label{small deformation open}
	Let $S_0\in M$ and  let $S$ be  a small deformation  $S_0$. Then  $S\in M$.
\end{pro}

Using Proposition \ref{small deformation open},  now we can prove the main theorem of this paper:

\begin{theorem}
	\label{connected}
	$\mathcal{M}$ is an irreducible  and  connected component of $\mathcal{M}^{5,2}_{1,1}$.
\end{theorem}
\begin{proof}
	Since $\mathcal{M}$ is the image of $M$  in $\mathcal{M}^{5,2}_{1,1}$, it is a constructible subset of $\mathcal{M}^{5,2}_{1,1}$, thus analytic openness (Proposition \ref{small deformation open}) implies Zariski openness. Therefore, $\mathcal{M}$ is a Zariski open and closed (Proposition \ref{closed}) subset of $\mathcal{M}^{5,2}_{1,1}$, hence it is a connected component of  $\mathcal{M}^{5,2}_{1,1}$.
\end{proof}

Considering Theorem \ref{one to one correspondence}, we also have the following:
\begin{cor}
	The canonical models of minimal surfaces  satisfying   condition $(\star)$  constitute an irreducible and connected   component of $\mathcal{M}^{5,2}_{1,1}$.
\end{cor}

To prove  proposition \ref{small deformation open}, we need the following two Lemmas.
\begin{lemma}
	\label{deformation of curves}
	Let  $p: \mathcal{S}\rightarrow \Delta$  be a smooth family of surfaces of general type  parametrized by a small disc $\Delta \subset \mathbb{C}$. Assume that for each $t\in \Delta$, there is an involution $\sigma_t$ on  $\mathcal{S}_t:=p^{-1}(t)$, which induce an involution $\sigma$ on $\mathcal{S}$. If   the fixed part $Fix(\sigma_0)$  of $\sigma_0$  has $n$ connected  components of dimension 1 and   $m$ isolated  points, then  $Fix(\sigma_t) (0\neq t\in \Delta)$ also  has $n$ connected  components of dimension 1 and    $m$ isolated  points.
\end{lemma}
\begin{proof} Let  $C_0^1,C_0^2,...,C_0^n$ be  the $n$ connected 1-dimensional components  of $Fix(\sigma)$ and $Q_0^1,...Q_0^m$ be the $m$ isolated    points of $Fix(\sigma_0)$. Take $n+m$ small open subsets $U_1,U_2,...,U_{n+m}$ on $\mathcal{S}$ such that $U_i\supset C_0^i$ $(1\leq i\leq n)$, $U_{n+i}\supset Q_i$ $(1\leq i\leq m)$ and $\bar{U_i}\cap \bar{U_j}=\emptyset$ for $i\neq j$. By choosing $\Delta$ small enough, we can assume that  $p|_{U_i}: U_i\rightarrow \Delta$ is surjective for $i=1,2...,n+m$.

	For $1\leq i\leq n$, take a point  $P_0^i\in C_0^i$;  for $n+1\leq n+i\leq n+m$,  let  $P_0^{n+i}:=Q_0^i$.  Choosing  a suitable  coordinate  system $(x,y,z)$ on $U_i$, we can assume $P_0^i=(0,0,0)$ and the action of $\sigma$ on $U_i$ is linear. Hence  the action    is
	(i) $(x,y,z)\mapsto (-x,y,z)$,   (ii)  $(x,y,z)\mapsto (-x,-y,z)$ or (iii) $(x,y,z)\mapsto (-x,-y,-z)$. In case (iii), $P_0^i$ is a singular point of $p$ (see \cite{Cai09} Lemma 1.4), contradicting our assumption that $p$ is smooth.

	In case (i),   $Fix(\sigma) \cap U_i$    is of dimension 2, thus it cannot be contained in $S_0$ since $\sigma|_{S_0}=\sigma_0$. In this case, $Fix(\sigma_0)\cap U_i=C^i_0$ and  $Fix(\sigma) \cap U\rightarrow \Delta$ is surjective,  hence there is  a connected component $C^i$ of $Fix(\sigma)\cap U_i$ that maps  surjectively to $\Delta$.

	In case  (ii), we have $p^*t=cz+$higher order terms.  Since $p$ is smooth, $c\neq0$. At $t=0$, the equation $p^*t=x=y=0$ has exactly one solution $(0,0,0)$ in $U_i$, thus  $P_0^i$ is an isolated fixed point of $\sigma_0$.  If we take  $\Delta$ and $U_i$ small enough, $p^*t=x=y=0$   has   one solution for any $t\in\Delta$. Thus $Fix(\sigma)\cap U_i= {\{x=y=0\}\cap U_i}\rightarrow \Delta$ is bijective.
	
	Now assume that for $0\neq t\in \Delta$,  $Fix(\sigma_t)$ has $n_t$ connected 1-dimensional components and $m_t$ isolated points, then we have  $n_t\geq n$, $m_t\geq m$. On the other hand,  since  we have a smooth  family $q:=p|_{Fix(\sigma)}: Fix(\sigma)\rightarrow \Delta$,  $Fix(\sigma_t)=q^{-1}(t)$ is smooth for each $t\in T$. By  the upper semi-continuity, we have  $n_t+m_t=h^0(\mathcal{O}_{Fix(\sigma_t)})\leq h^0(\mathcal{O}_{Fix(\sigma_0)})=n+m$.  Therefore $n_t=n$ and $m_t=m$.
\end{proof}

\begin{remark}
	\label{remark for deformation of curves}
	If we replace the smooth family $p:\mathcal{S}\rightarrow \Delta$ with   the  flat family $p': \mathcal{S}'\rightarrow \Delta$ (here $\mathcal{S}_t':=p'^{-1}(t)$ is the canonical model of $\mathcal{S}_t$)  in the above 	lemma,   using a similar argument, one can show:
	if $Fix(\sigma_0)$ contains  $n$ smooth connected 1-dimensional components, then
	$Fix(\sigma_t)$ $(0\neq t\in \Delta)$ also contains  $n$ smooth connected 1-dimensional components.
\end{remark}

\begin{lemma}
	\label{3 independet sections}
	Let $V$ be  a rank 3 vector bundle over an elliptic curve $B$. If  the total space $\mathbb{P}(V^\lor)$ (sometimes we just write $\mathbb{P}(V)$ if   no confusion) of $V$ has three independent sections $s_i: B\rightarrow \mathbb{P}(V^\lor)$(`independent' means for any fibre $F$ of $\mathbb{P}(V^\lor)\rightarrow B$, the three points $s_i(B)\cap F$ are not contained in any line in $F$), then $V$ is a direct sum of three line bundles.
\end{lemma}
\begin{proof} Denote by $V_b$ the affine 3-space of the restriction of $V$ to $b\in B$. Let $P_b^i:=s_i(B)\cap V_b$. Choose a coordinate system  $(x_b,y_b,z_b)$ for $V_b$ and  assume $P_b^i=(x_b^i,y_b^i,z_b^i)$. Since $\{P_b^i\}_{i=1,2,3}$ are not contained in any line in $V_b$, at each point $b\in B$, the three subspaces $\mathbb{C}(x_b^i,y_b^i,z_b^i)$ $(i=1,2,3)$ of $V_b$ span $V_b$. Thus the three  independent sections $s_i$ give three  sub-(line)-bundles $N^i(i=1,2,3)$ $(N^i_b=\mathbb{C}(x_b^i,y_b^i,z_b^i))$ of $V$, which generate $V$ over each point $b\in B$. Hence  $V$ is a direct sum of three  line bundles $N^i$ $(i=1,2,3)$.
\end{proof}

Now we   are in the situation to  prove Proposition \ref{small deformation open}.
\begin{proof}[Proof of Proposition \ref{small deformation open}]
	
	Let $p': \mathcal{S}' \rightarrow \Delta$ be a flat family of canonical models of surfaces of general type with $\mathcal{S}_0':=p'^{-1}(0)=S_0'$, where $S'_0$ is the canonical model of $S_0$. Taking    a base change (for simplicity we still denote by $\Delta$) and taking the simultaneous resolution, we have  a smooth family of minimal surfaces $p: \mathcal{S}\rightarrow \Delta$ with $\mathcal{S}_0:=p^{-1}(0)=S_0$.

	By Lemma \ref{limit V1 V2}, for $0\neq t\in \Delta$, $\mathcal{S}_t$ is  a minimal surface with $p_g=q=1, K^2=5,g=2$  and $V_1$   indecomposable. In particular, $\mathcal{S}_t$ has an   involution $\sigma_t$ induced by the Albanese fibration, which induces an   involution  $\sigma'_t$ on $\mathcal{S}'_t$.  The involution $\sigma'_t$ on each $\mathcal{S}'_t$ induces an involution $\sigma'$ on $\mathcal{S}'$. Let $\mathcal{C}:=\mathcal{S}'/{\sigma'}$, then we have a  flat family $\hat{p}: \mathcal{C}\rightarrow \Delta$.

	By Remark \ref{remark for branch divisors}, $Fix(\sigma'_0)$    contains  four  smooth sections.  By Remark \ref{remark for deformation of curves},  for $0\neq t\in \Delta$ $Fix(\sigma'_t)$ also contains four smooth sections, hence the branch curve of the double cover $\mathcal{S}'_t\rightarrow \mathcal{C}_t:=\hat{p}^{-1}(t)$ contains four smooth sections.

	\vspace{3ex}
	{\em Claim}: For $0\neq t\in \Delta$, $\mathbb{P}(V_2)$ has three  independent sections. Therefore $V_2$ is a direct sum of three  line bundles by    Lemma \ref{3 independet sections}.
	
	{\em Proof of the claim}: now we have a flat family $\hat{p}: \mathcal{C}\rightarrow \Delta$ of conic bundles over elliptic curves.
	Note that the smooth fibre of $\mathcal{C}_t\rightarrow B_t$ is a smooth conic in   $F$, and  any three of the four smooth sections intersect with $F$ at three distinct  points lying on the conic, thus they   are  not contained in any line in $F$. So we only need to consider the  singular  fibres of $\mathcal{C}_t$. Since $\mathcal{C}_0$ has only three  singular fibres (see  Remark \ref{remark for branch divisors}), for $0\neq t\in \Delta$, $\mathcal{C}_t$ has at most   three  singular fibres.
	
	Since    each singular fibre of $\mathcal{C}_0$  is a union of two distinct lines $L_0^1,L_0^2$, we see  that   for  $0\neq t\in \Delta$, each singular  fibre of $\mathcal{C}_t$ is also a  union of two   distinct lines $L_t^1,L_t^2$.
	Note that on $\mathcal{C}_0$, two of the  four smooth sections intersect only with $L_0^1$ and  the other two smooth sections intersect only with $L_0^2$, w.l.o.g. we can assume that $C_0^1,C_0^2$ intersect with $L_0^1$ and   $C_0^3,C_0^4$ intersect with $L_0^2$.  Since $C_0^i$ and $C_0^j$ $(j\neq i)$ are disjoint, using  a similar argument as Lemma \ref{deformation of curves}, for small $\Delta$, sections $C_t^1,C_t^2$ do not intersect with $L_t^2$, and sections  $C_t^3,C_t^4$ do not intersect with $L_t^1$.  Hence for $0\neq t\in \Delta$,  $C_t^1,C_t^2$ intersect with $L_t^1$ and $C_t^3,C_t^4$ intersect with $L_t^2$.  Thus any three  of the four sections intersect with  the singular fibre $L_t$    at three  points that are  not contained in any line in $F$. Therefore  for $0\neq t\in \Delta$, $\mathbb{P}(V_2)$ also has three  independent sections.
	
	\vspace{3ex}
	We have proved that for any $t\in \Delta$,  $V_2$ is a direct sum of three  line bundles. Now we use a similar  argument as  Lemma \ref{limit V1 V2} to show that each direct summand of $V_2$ is $\mathcal{O}_B(2\cdot 0)$:
	
	since for $t=0$, $h^0(V_2(-2p))=0$ for any $p\neq 0$,  shrinking $\Delta$ and using  the upper semi-continuity, we see  that for $0\neq t\in \Delta$,  $h^0(V_2(-2p))=0$ for any $p\neq 0$. Since $V_2$ is a direct sum of three  line bundles, this happens if and only if
	$V_2=\mathcal{O}_B(2\cdot0)\oplus \mathcal{O}_B(2\cdot0)\oplus \mathcal{O}_B(2\cdot0)$.
	
	{\color{blue}} Therefore for any $0\neq t\in \Delta$, $\mathcal{S}_t$ satisfies  condition $(\star)$. By Theorem \ref{one to one correspondence}, we conclude that  $\mathcal{S}_t\in M$.
\end{proof}

\section{The number of  direct summands of $f_*\omega_S^{\otimes 2}$ is not a deformation invariant}
Let $S$ be a minimal surface of general type with $p_g=q=1$ and let $f:S\rightarrow B:=Alb(S)$ be the Albanese fibration of $S$. Let $g$ be the genus of a general Albanese fibre. Set $V_n:=f_*\omega_S^{\otimes n}$ and denote by $\nu_n$ {\em the number of direct summands} of $V_n$. Catanese-Ciliberto \cite{CC91} proved that $\nu_1$ is a topological invariant, hence it is also a deformation invariant. In this section we show that $\nu_2$ is not a deformation invariant, which  gives a negative answer to Pignatelli's question (cf. \cite{Pig09} p. 3).

\vspace{3ex}
For later convenience, we fix a group structure for the genus one curve $B=Alb(S)$, denote  by $0$ its neutral element and  by $\tau$ a nontrivial 2-torsion point.    Let $E_{[0]}(2,1)$ be the unique  indecomposable vector bundle of rank two on $B$ with $\det E_{[0]}(2,1)\cong \mathcal{O}_B(0)$ (cf. \cite{Ati57}).

Denote by $\mathcal{M}_I^{3,2}, \mathcal{M}_{II}^{3,2}, \mathcal{M}_{III}^{3,2}$ the  subsets of $\mathcal{M}_{1,1}^{3,2}$   corresponding to surfaces with $p_g=q=1,K^2=3,g=2$ such that   $\nu_2=1, 2, 3$ respectively (cf. \cite{CP06} Definition 6.11). Then  we have $\mathcal{M}_{1,1}^{3,2}=\mathcal{M}_I^{3,2}\cup \mathcal{M}_{II}^{3,2} \cup \mathcal{M}_{III}^{3,2}$.

\vspace{3ex}
The main ingredient to prove that $\nu_2$ is not a deformation invariant  is the following
\begin{theorem}
	\label{M2 not empty}
	There exist minimal  surfaces with $p_g=q=1, K^2=3, g=2$ such that  $V_2=E_{[0]}(2,1)(0)\oplus \mathcal{O}_B(\tau)$,  i.e., $\mathcal{M}_{II}^{3,2}$ is not empty.
\end{theorem}

We shall prove Theorem \ref{M2 not empty} later. First   we show how   Theorem  \ref{M2 not empty}   gives a negative answer to  Pignatelli's question.
\begin{cor}
	\label{number v2 not a deformation invariant}
	$\nu_2$ is not a deformation invariant, hence  it is  not a topological invariant, either.
\end{cor}

\begin{proof}
	
	Catanese-Pignatelli  (cf. \cite{CP06} Proposition 6.3)  proved  that  $\mathcal{M}_{1,1}^{3,2}$ has exactly three  irreducible connected components: one  is  $\overline{\mathcal{M}_I^{3,2}}$ and two are contained in   $\overline{\mathcal{M}_{III}^{3,2}}$. By Theorem \ref{M2 not empty},  $\mathcal{M}_{II}^{3,2}$ is not empty.  Hence  either $\mathcal{M}_{II}^{3,2}\cap\overline{\mathcal{M}_I^{3,2}}$ or $\mathcal{M}_{II}^{3,2}\cap\overline{\mathcal{M}_{III}^{3,2}}$  is nonempty.
	In particular,  there is a minimal surface  with $\nu_2=2$  that can be deformed to a minimal surface  with  $\nu_2=1$ or $\nu_2=3$. Therefore  $V_2$ is not a deformation invariant.
	
	( We will show in the following Remark \ref{belong to M1}  that   the minimal surfaces we constructed in Theorem \ref{M2 not empty} belong to $\overline{\mathcal{M}_I}$. )
\end{proof}

Now we prove Theorem \ref{M2 not empty}.
\begin{proof}[Proof of Theorem \ref{M2 not empty}]
	Let $B$ be an elliptic curve and $N:=\mathcal{O}_B(\tau-0)$ be the  torsion line bundle of order 2 on $B$. Let  $V'_1:=E_{[0]}(2,1)$ and  $V'_2:=E_{[0]}(2,1)(0)\oplus N(0)$  be two vector bundles on $B$.  To  prove Theorem \ref{M2 not empty},  it suffices to show that
	there exists a relatively  minimal genus 2 fibration   $f:S\rightarrow B$ such that   $p_g(S)=q(S)=1, K_S^2=3, g=2$, $V_1=f_*\omega_B=V'_1$,  $V_2=f_*\omega_B^{\otimes 2}=V'_2$. (By the universal property of    Albanese map, $f$ must be the Albanese fibration of $S$.)

	By Catanese-Pignatelli's structure theorem for genus 2 fibrations (cf.  \cite{CP06} section 4),
	it suffices to find  a conic bundle $\mathcal{C}\in |\mathcal{O}_{\mathbb{P}(V'_2)}(2)\otimes \pi^*\det(V'_1)^{-2}|$ on $\mathbb{P}(V'_2)$ and an effective  divisor $\delta\in |\mathcal{O}_{\mathcal{C}}(3)\otimes \pi^*\mathcal{O}_B(-2\cdot 0-2\tau)|$ such that $\mathcal{C}$ contains exactly  one RDP  as singularities, $\delta$ does not contain the singular point of $\mathcal{C}$,  and the double cover  $X$ of $\mathcal{C}$ with branch divisor $\delta$ has at most RDP's as singularities.
	
	\vspace{3ex}
	To get global relative coordinates on the fibre of $\mathbb{P}(V'_2)$, we
	take  an unramified double covering  $\phi: \tilde{B}\rightarrow B $  such that $\phi^*N\cong \mathcal{O}_{\tilde{B}}$ and  $\phi^*0=\tilde{0}+\eta$ for some nontrivial  2-torsion point $\eta\in \tilde{B}$, where $\tilde{0}$ is the neutral element in the group structure of $\tilde{B}$ such that $\phi(\tilde{0})=0$. By \cite{Ish05} Theorem 2.2, Lemma 2.3,  we have
	$\phi^*E_{[0]}(2,1) \cong \mathcal{O}_{\tilde{B}}(p)\oplus \mathcal{O}_{\tilde{B}}(p')$, where $\mathcal{O}_B(\phi_*(p)-0)\cong N$ (cf. \cite{Fri98} Chapter 2, Proposition 27) and $p'=p\oplus \eta$ in the group law of $\tilde{B}$.
	
	Now let $\tilde{E}:=\phi^*(E_{[0]}(2,1)\oplus N)$, then we have the following commutative diagram:
	$$\xymatrix{\mathbb{P}(\tilde{E})\ar[rrr]^{\Phi}\ar[d]^{\tilde{\pi}}& &&\mathbb{P}(E_{[0]}(2,1)\oplus N)\ar[d]^{\pi} \\
		\tilde{B}\ar[rrr]^{\phi} &&&B}
	$$
	where $\tilde{\pi}: \mathbb{P}(\tilde{E})\rightarrow \tilde{B}$ is the natural $\mathbb{P}^2$-bundle over $\tilde{B}$.
	Note that the unramified double cover $\Phi: \mathbb{P}(\tilde{E})\rightarrow \mathbb{P}(E_{[0]}(2,1)\oplus N)\cong \mathbb{P}(V'_2)$ induces an involution  on $\mathbb{P}(\tilde{E})$, which we also denote by $T_\eta$. Let $G:=<T_\eta>$, then $G$ acts on $\tilde{B}$ and $\mathbb{P}(\tilde{E})$ effectively.
	
	Now to find  a conic bundle $\mathcal{C}\in |\mathcal{O}_{\mathbb{P}(V'_2)}(2)\otimes \det(V'_1)^{-2}|=|\mathcal{O}_{\mathbb{P}(E_{[0]}(2,1)\oplus N)}(2)|$ containing exactly one RDP as singularity, is equivalent  to finding   a G-invariant conic $\mathcal{\tilde{C}}\in |\mathbb{P}(\tilde{E})(2)|$ on $\mathbb{P}(\tilde{E})$, which  contains exactly 2 RDP's on two different fibres as singularities. Similarly, to find  a curve $\delta \in |\mathcal{O}_{\mathcal{C}}(3)\otimes \pi^*(-2\cdot0-2\tau)|$ on $\mathcal{C}$ such that $\delta$ does not contain the singular point of $\mathcal{C}$ and  the double cover  $X$ of $\mathcal{C}$ with branch divisor  $\delta$ has at most RDP'S as singularities,
	is equivalent  to finding  a G-invariant curve $\tilde{\delta} \in |\mathcal{O}_{\mathcal{\tilde{C}}}(3)\otimes \tilde{\pi}^*(-0-\eta)|$ on $\mathcal{\tilde{C}}$ such that $\tilde{\delta}$ does not contain the singularities of $\mathcal{\tilde{C}}$, and    the double cover $\tilde{X}$ of $\tilde{\mathcal{C}}$ branched on $\tilde{\delta}$ has at most RDP'S as singularities.

	Take relative coordinates $y_1:\mathcal{O}_{\tilde{B}}(p) \rightarrow \tilde{E}$, $y_2:\mathcal{O}_{\tilde{B}}(p') \rightarrow \tilde{E}$, $y_3:\mathcal{O}_{\tilde{B}} \rightarrow \tilde{E}$ on the fibre of  $\mathbb{P}(\tilde{E})$.    Then the action of $T_\eta^*$ is just: $y_1\mapsto y_2, y_2\mapsto y_1$ and $ y_3\mapsto -y_3$.
	Let  $\mathcal{\tilde{C}}\subset\mathbb{P}(\tilde{E})$ be the conic bundle defined by $$f=a_1^2y_1^2+a_2^2y_2^2+a_3y_3^2=0,$$ where $a_1\in H^0(\mathcal{O}_{\tilde{B}}(p)), a_2=T_\eta^*a_1\in H^0(\mathcal{O}_{\tilde{B}}(p')), a_3\in H^0(\phi^*\mathcal{O}_B)^G=H^0(\mathcal{O}_{\tilde{B}})$,  $a_3\neq 0$. It is easy to see  that $\tilde{\mathcal{C}}$ is G-invariant.  Since $$(\frac{\partial f}{\partial y_1}, \frac{\partial f}{\partial y_2}, \frac{\partial f}{\partial y_3})=(2a_1^2y_1, 2a^2_2y_2, 2a_3y_3),$$
	the only possible singularities of $\tilde{\mathcal{C}}$ are: $P_1: y_1= y_3=a_2=0$ and $P_2: y_2=y_3=a_1=0$. It is easy to check that $P_1$ is a $A_1$-singularity on the fibre of  $\tilde{\pi}|_{\tilde{\mathcal{C}}}$ over $p' \in \tilde{B}$, and $P_2$ is a $A_1$-singularity on the fibre of  $\tilde{\pi}|_{\tilde{\mathcal{C}}}$  over $p \in \tilde{B}$.

	Considering \cite{CP06} Lemma 6.14,  we take $\tilde{\delta}$ as the  complete intersection of $\mathcal{\tilde{C}}$ with a relative cubic $\tilde{\mathcal{G}} \in |\mathcal{O}_{\mathbb{P}(\tilde{E})}(3)\otimes \tilde{\pi}^*\mathcal{O}_{\tilde{B}}(-\tilde{0}-\eta)|$.
	Let $\Delta$ be the linear subspace of  $|\mathcal{O}_{\mathbb{P}(\tilde{E})}(3)\otimes \tilde{\pi}^*\mathcal{O}_{\tilde{B}}(-\tilde{0}-\eta)|$ consisting of divisors defined by the equations:
	$$g=b_1y_1^3+b_2y_2^3+b_3y_1y_2y_3,$$
	where $b_1\in H^0(\mathcal{O}_{\tilde{B}}(p')\otimes \phi^*N), b_2=T_\eta^*b_1\in H^0(\mathcal{O}_{\tilde{B}}(p)\otimes \phi^*N)$ , $b_3\in H^0(\phi^*N)$, $b_3\neq 0$.
Since $T_\eta^*b_3=-b_3, T_\eta^*y_3=-y_3$,  an element $\tilde{\mathcal{G}} \in \Delta$ is G-invariant. Moreover, $\tilde{\mathcal{G}} $ does not contain the singularities $P_1, P_2$ of $\tilde{\mathcal{C}}$.
	
	When $b_1,b_3$ vary, $\Delta$ has no fixed points except the 4 curves $C_1:=\{y_1=y_2=0\}$, $C_2:=\{b_1=y_2=0\}$(on the fibre over $p'$), $C_3:=\{b_2=y_1=0\}$(on the fibre over $p$) and $C_4:= \{y_3=b_1y_1^3+b_2y_2^3=0, y_1\neq0,y_2\neq0\}$.  Note that   $C_1$ does not intersect  $\tilde{\mathcal{C}}$, so for a general member $\tilde{\mathcal{G}} \in \Delta$, $\tilde{\delta}=\tilde{\mathcal{G}} \cap \tilde{\mathcal{C}}$ is smooth outside $(C_2\cup C_3\cup C_4)\cap\tilde{\mathcal{C}}$.
	
	Now we show that $\tilde{\delta}$ is smooth at $(C_2\cup C_3\cup C_4)\cap \tilde{\mathcal{C}}$  by computing  the rank of  the Jacobian matrix. For $C_2\cap \tilde{\mathcal{C}}:=\{b_1=y_2=a_1^2y_1^2+a_3y_3^2=0\}$, the Jacobian  matrix is
	\begin{gather*}
	\begin{pmatrix}
	0 & 2a_1^2y_1(\neq 0) & 0 & 2a_3y_3 \\
	ky_1^3(\neq 0)& 0 & 0 & 0
	\end{pmatrix}
	\end{gather*}
	which has rank 2, therefore $\tilde{\delta}$ is smooth at $C_2 \cap \tilde{\mathcal{C}}$.  The proof for $C_3\cap \tilde{\mathcal{C}}$ is similar (since  $C_2,C_3$ are symmetric).
	
	For $C_4\cap \tilde{\mathcal{C}}=\{y_3=a_1^2y_1^2+a_2^2y_2^2=b_1y_1^3+b_2y_2^3=0, y_1\neq0,y_2\neq0 \}$, the Jacobian matrix is
	\begin{gather*}
	\begin{pmatrix}
	0 & 2a_1^2y_1(\neq 0) & 2a_2^2y_2(\neq 0) & 0 \\
	0& 3b_1y_1^2(\neq 0)& 3b_2y_2^2(\neq 0) & b_3y_1y_2(\neq0)
	\end{pmatrix}
	\end{gather*}
	which has rank 2, thus $\tilde{\delta}$ is smooth at $C_4\cap\tilde{\mathcal{C}}$.
	Hence  for a general member $\tilde{\mathcal{G}}  \in \Delta$, $\tilde {\delta}=\tilde{\mathcal{G}} \cap \tilde{\mathcal{C}}$ is smooth. Therefore, the double cover of $\tilde{\mathcal{C}}$ with branch divisor $\tilde{\delta}$ is smooth.
	
	\vspace{3ex}
	Let $\mathcal{C}:=\Phi(\tilde{\mathcal{C}}), \mathcal{G}:=\Phi(\tilde{\mathcal{G}})$ and  $\delta:=\mathcal{C}\cap \mathcal{G}$. By \cite{CP06} Theorem 4.13,  the double cover $S\rightarrow \mathcal{C}$ with branch divisor $\delta$  is a smooth double cover,  and $S$ is a   minimal surface with  $p_g=q=1,K^2=3,g=2$, $V_1=V'_1=E_{[0]}(2,1)$ and $V_2=V'_2=E_{[0]}(2,1)(0)\oplus \mathcal{O}_B(\tau)$.
\end{proof}

\begin{remark}
	\label{belong to M1}
	In fact, the minimal surfaces   constructed in Theorem \ref{M2 not empty} are  contained in $\overline{\mathcal{M}_I^{3,2}}$.
\end{remark}
\begin{proof}
	(1) For general choices of $\mathcal{C}\in |\mathcal{O}_{\mathbb{P}(V_2)}(2)\otimes \pi^*\det(V_1)^{-2}|$ and $\mathcal{G}\in  |\mathcal{O}_{\mathbb{P}(V_2)}(3)\otimes \pi^*\mathcal{O}_B(-2\cdot 0-2\tau)|$, $\delta=\mathcal{C}\cap \mathcal{G}$ is connected:
	
	Since we have proved that  for general choices of $\mathcal{C}$ and  $G$,  $\delta$ is smooth, it suffices to show  $h^0(\mathcal{O}_{\delta})=1$. Let $\pi: W:=\mathbb{P}(V_2)\rightarrow B$ be the natural projective bundle, $T$ be the divisor on $W$ such that $\pi_*\mathcal{O}_W(T)=V_2(-0)$, and $H_t$ be the fibre of $\pi$ over $t\in B$.

	Consider the following exact sequences
	$$0\rightarrow \mathcal{O}_W(-5T+H_0)\rightarrow  \mathcal{O}_W(-3T+H_0)\rightarrow \mathcal{O}_{\mathcal{C}}(-3T+H_0)\rightarrow 0$$
	$$0\rightarrow  \mathcal{O}_{\mathcal{C}}(-3T+H_0)\rightarrow  \mathcal{O}_{\mathcal{C}} \rightarrow \mathcal{O}_{\delta}\rightarrow 0$$
	By \cite{BHPV} Chap. I, Theorem 5.1 and using Serre duality, we have $h^0( \mathcal{O}_W(-3T+H_0))=h^0(\pi_* \mathcal{O}_W(-3T+H_0))=0$, $h^1( \mathcal{O}_W(-3T+H_0))=h^2(\mathcal{O}_W(-H_0+H_\tau))=h^2(\mathcal{O}_B(-0+\tau))=0$, $h^1( \mathcal{O}_W(-5T+H_0))=h^2(\mathcal{O}_W(2T-H_0+H_\tau))=h^2(\pi_*(\mathcal{O}_W(2T-H_0+H_\tau))=0$, $h^2(\mathcal{O}_W(-5T+H_0))=h^1(\mathcal{O}_W(2T-H_0+H_\tau))=h^1(S^2(V_1)(\tau-0))=0$. Hence we have $h^1(\mathcal{O}_{\mathcal{C}}(-3T+H_0) )=0$.
	Since moreover $h^0(\mathcal{O}_{\mathcal{C}} )=1$, we get 	 $h^0(\mathcal{O}_{\delta})=1$.

	(2)
	$S$ is    not contained in   $\overline{\mathcal{M}_{III}^{3,2}}$.
	If these surfaces  were  contained in $\overline{\mathcal{M}_{III}^{3,2}}$, then   we get a 1-parameter connected flat family $\mathcal{S}\rightarrow T$ of  canonical models of minimal surfaces  with $p_g=q=1, K^2=3, g=2$ such that the central fibre has $\nu_2=2$ while  a general fibre has $\nu_2=3$.  Now we have a flat  family of  double covers of conic bundles having only RDP's as singularities such that, for a general fibre  the branch curve is reducible and disconnected (see \cite{CP06} proposition 6.16), hence $h^0(\delta_t)>1 (t\neq0)$;  while for  the central fibre  the branch curve is irreducible and smooth, hence  $h^0(\delta_0)=1<h^0(\delta_t)(t\neq 0)$,  contradicting   the upper semi-continuity.

	By the proof of   Corollary \ref{number v2 not a deformation invariant},  we  see  that $S$ is contained in $\overline{\mathcal{M}_I^{3,2}}$.
\end{proof}

\vspace{3ex}
 $\mathbf{Acknowledgements}.$
 The author is sponsored by China Scholarship Council ``High-level university graduate program''. 
 
The author would like to thank his advisor, Professor Fabrizio Catanese at Universit\"at Bayreuth, for suggesting this research topic,  for a lot of  inspiring discussion with the author  and for his encouragement  to the author.  The author would also like to thank his domestic advisor, Professor Jinxing Cai  at Peking University, for his encouragement and some useful suggestions.
The author is grateful to Binru Li and Roberto  Pignatelli for a lot of helpful discussion.

	%\bibliography{/Users/sihong/Reference/journalname,/Users/sihong/Reference/graph}

\vspace{3ex}
School of Mathematical Sciences, Peking  University, Yiheyuan Road 5, Haidian District, Beijing 100871,  People's Republic of China

E-mail address: 1201110022@pku.edu.cn

\vspace{2ex}
Lehrstuhl Mathematik VIII, Universit\"at Bayreuth,  NW II,  Universit\"atsstr. 30, 95447 Bayreuth, Germany

E-mail address: songbo.ling@uni-bayreuth.de

\end{document}